\documentclass{amsart}

\usepackage[colorlinks=true, urlcolor=blue, citecolor=blue, linkcolor=blue, hypertexnames=false]{hyperref}
\usepackage{amssymb}
\usepackage{mathrsfs}
\usepackage{dsfont}
\usepackage{mathtools}
\usepackage{aliascnt}

\newtheorem{qst}{Question}

\newtheorem{thm}{Theorem}

\newtheorem{lem}{Lemma}

\newtheorem{prp}{Proposition}

\newtheorem{cor}{Corollary}

\theoremstyle{definition}

\newtheorem{dfn}{Definition}

\newtheorem{xpl}{Example}

\numberwithin{equation}{section}

\author{Charles A. Akemann \lowercase{and} Tristan Bice}
\address{University of California, Santa Barbara, United States\newline
\lowercase{and} Federal University of Bahia, Salvador, Brazil}
\title{Hereditary C*-Subalgebra Lattices}
\keywords{C*-Algebras, Hereditary C*-Subalgebras, *-Annihilators, Lattices, Ortholattices, Topology, Locales, Quantales, Type Decomposition}
\subjclass[2010]{Primary: 46L05, 46L85; Secondary: 06B23, 06C15}

\begin{document}

\begin{abstract}
We investigate the connections between order and algebra in the hereditary C*-subalgebra lattice $\mathcal{H}(A)$ and *-annihilator ortholattice $\mathscr{P}(A)^\perp$.  In particular, we characterize $\vee$-distributive elements of $\mathcal{H}(A)$ as ideals, answering a 25 year old question, allowing the quantale structure of $\mathcal{H}(A)$ to be completely determined from its lattice structure.  We also show that $\mathscr{P}(A)^\perp$ is separative, allowing for C*-algebra type decompositions which are completely consistent with the original von Neumann algebra type decompositions.
\end{abstract}

\maketitle

\section{Introduction}

\subsection{Motivation}

Hereditary C*-subalgebras $\mathcal{H}(A)$ of a C*-algebra $A$ have long been considered analogs of open sets.  Given the fundamental role open subsets and their lattice structure play in topological spaces (as more clearly seen in the point-free topology of frames and locales), one would expect us by now to have a deep understanding of $\mathcal{H}(A)$, with numerous theorems relating algebraic properties of $A$ to order properties of $\mathcal{H}(A)$.  But on the contrary, our knowledge of $\mathcal{H}(A)$ is still quite limited, and the study of $\mathcal{H}(A)$ has remained very much on the periphery of mainstream C*-algebra research.  Needless to say, we see this as a somewhat strange state of affairs.

Another perplexing trend in operator algebras is the early divergence of von Neumann algebra and C*-algebra theory.  Again, one would naturally expect that, as von Neumann algebras form a nice subclass of C*-algebras, much inspiration could be drawn from looking at the von Neumann algebra theory and trying to generalize it in various ways to C*-algebras.  But yet again, we rarely see this happening, particularly in modern C*-algebra research, with much of the von Neumann algebra theory dismissed long ago as either inapplicable or irrelevant to general C*-algebras.

In fact, this is no coincidence, as it is precisely this more topological, order theoretic approach that is required to generalize some of the basic von Neumann algebra theory.  This can be seen in \cite{Bice2014c} and \cite{Bice2014b}, and we continue in this direction in the present paper, using mainly classical theory to prove a number of new, fundamental and very general C*-algebra results regarding $\mathcal{H}(A)$, and its subset of *-annihilators $\mathscr{P}(A)^\perp$.  We hope this might spur on further research in this largely neglected subfield of C*-algebra theory.

\subsection{Outline}

We give the necessary basic definitions and assumptions for the rest of the paper in \S\ref{prelim}, which the reader is welcome to skim over and refer back to only when unfamiliar terminology or notation appears later on.  We start the paper proper with a brief note on compactness in \S\ref{C&C}.  Following that, we examine the semicomplement structure of $\mathcal{H}(A)$ in \S\ref{S&O} and obtain various characterizations of strong orthogonality in \autoref{trieq}.  Next, in \S\ref{C&P}, we exhibit a correspondence, in the unital case, between two of the most common objects that appear in lattices and C*-algebras, namely complements and projections.  We then show in \S\ref{*AI} that the superficially similar notion of a $\wedge$-pseudocomplement turns out to have a quite different algebraic characterization in $\mathcal{H}(A)$, namely as an annihilator ideal.  Next, in \S\ref{D&I}, we show that arbitrary ideals in $\mathcal{H}(A)$ can be characterized as the $\vee$-distributive elements, answering a long-standing question from \cite{BorceuxRosickyBossche1989}.  Quantales, as introduced in \cite{Mulvey1986}, have often been considered the appropriate non-commutative analogs of locales, and this characterization shows that the natural quantale structure on $\mathcal{H}(A)$ is, in fact, completely determined by its lattice structure.

Another result of fundamental importance is the fact that *-annihilators satisfy a strong version of the SSC property, which we show in \S\ref{S&*A} (although, as we show in \autoref{C01K}, this property does not quite characterize the *-annihilators).  We then examine the *-annihilators $\mathscr{P}(A)^\perp$ as an ortholattice in its own right in \S\ref{*AO}.  It turns out that many of the order theoretic concepts examined in $\mathcal{H}(A)$ converge in the case of $\mathscr{P}(A)^\perp$, coinciding with the annihilator ideals, as we show in \autoref{anncentre}.  Lastly, we outline in \S\ref{C*TD} how the separativity/SSC property of $\mathscr{P}(A)^\perp$ allows for order theoretic type decompositions that are completely consistent with the original von Neumann algebra type decompositions.

\section{Preliminaries}\label{prelim}

\subsection{Posets}

Recall that a \emph{poset} is simply a partially ordered set, i.e. a set $\mathbb{P}$ together with a binary relation $\leq$ on $\mathbb{P}$ that is transitive ($p\leq q\leq r\Rightarrow p\leq r$), reflexive ($p\leq p$) and anti-symmetric ($p\leq q\leq p\Rightarrow p=q$).  A \emph{subposet} of $\mathbb{P}$ is simply a subset of $\mathbb{P}$ under the same ordering.  For any poset $\mathbb{P}$ there is a dual poset $\mathbb{P}^*$ such that the underlying set is the same but the order reversed.  Consequently, all poset definitions and results also have duals.

\emph{Infimums} (greatest lower bounds) are, when they exist, denoted by $\bigwedge S$ (in particular, $\bigwedge\emptyset$ is the maximum of $\mathbb{P}$, when it exists).  We also write $p\wedge q$ for $\bigwedge\{p,q\}$.  A \emph{$\wedge$-lattice} (or \emph{meet semilattice}) is a poset $\mathbb{P}$ where $p\wedge q$ exists, for all $p,q\in\mathbb{P}$.  A \emph{$\wedge$-sublattice} of a $\wedge$-lattice $\mathbb{P}$ is a subset $S$ closed under $\wedge$.  We similarly define $\bigwedge$-(sub)lattices.  Dually, we denote \emph{supremums} (least upper bounds) by $\bigvee$ and $\vee$ and define $\vee$-(sub)lattices and $\bigvee$-(sub)lattices accordingly.  In fact, $\mathbb{P}$ is a $\bigwedge$-lattice precisely when $\mathbb{P}$ is a $\bigvee$-lattice, in which case $\mathbb{P}$ is called a \emph{complete} lattice.  For the remainder of this paper, let us assume that
\begin{center}\textbf{$\mathbb{P}$ is a poset with $0=\bigwedge\mathbb{P}$ and $1=\bigvee\mathbb{P}$.}\end{center}

Given a set $X$, we always order its subsets $\mathscr{P}(X)$ by $\subseteq$ (inclusion), which makes $\mathscr{P}(X)$ a complete lattice in which $\bigvee$ is $\bigcup$ (union) and $\bigwedge$ is $\bigcap$ (intersection).  Many of the posets we consider in this paper are subposets of $\mathscr{P}(X)$, for some $X$, like the open subsets $\mathscr{P}(X)^\circ$ of a topological space $X$.  Here $\mathscr{P}(X)^\circ$ is, by definition, nothing more than some $\bigvee$-sublattice and $\wedge$-sublattice of $\mathscr{P}(X)$ containing $\{\emptyset,X\}$ (in fact, as $\mathscr{P}(X)^\circ$ is a $\bigvee$-sublattice, we automatically have $\emptyset=\bigvee\emptyset\in\mathscr{P}(X)^\circ$).

For another poset central to this paper, consider a C*-algebra $A$, i.e. a Banach *-algebra with $||a^*a||=||a||^2$, for all $a\in A$.  The \emph{positive} elements $A_+$ of $A$ are precisely those of the form $a^*a$, for some $a\in A$.  We order $A_+$ in the usual way by
\[a\leq b\quad\Leftrightarrow\quad b-a\in A_+.\]
As in \cite{Pedersen1979} \S1.5, we call a C*-subalgebra $B$ of $A$ \emph{hereditary} when, for all $a,b\in A_+$,
\begin{equation}
a\leq b\in B\quad\Rightarrow\quad a\in B,
\end{equation}
although they could equivalently be defined as closed self-adjoint bi-ideals (see \cite{Blackadar2013} Proposition II.3.4.2 and Corollary II.5.2.9) or *-annihilators with respect to the canonical action of $A$ on $A^*$ (see \cite{Effros1963} Theorem 2.5).  These hereditary C*-subalgebras $\mathcal{H}(A)$ form a complete lattice\footnote{which is isomorphic to many other lattices defined from $A$, e.g. the lattice of closed left or right ideals in $A$, closed cones in $A_+$ (see \cite{Effros1963} Theorem 2.4), norm filters in (unital) $A^1_+$ (see \cite{Bice2011} Corollary 3.4), weak* closed faces of $A^{*1}_+$ (see \cite{Pedersen1979} Proposition 3.11.9), or open projections in $A''$ (see \cite{Akemann1969} Proposition II.2).  Indeed, we will often work with open projections, as in \cite{Akemann1969}, \cite{Akemann1970} and \cite{Akemann1971}, but $\mathcal{H}(A)$ has the advantage that each $B\in\mathcal{H}(A)$ remains in the category of C*-algebras, so concepts like commutativity naturally carry over.} in which $\bigwedge$ is $\bigcap$.

\subsection{Commutative C*-Algebras}\label{comC*}

From now on, we assume that
\begin{center}\textbf{$X$ is a topological space.}\end{center}
We denote the continuous and bounded continuous functions from $X$ to $\mathbb{C}$ by $C(X)$ and $C^b(X)$ respectively.  With pointwise operations and the supremum norm, $C^b(X)$ becomes a commutative C*-algebra.  For any $Y\subseteq X$, we have a hereditary C*-subalgebra of $C^b(X)$ given by
\[B_Y=\{f\in C^b(X):f[X\setminus Y]=\{0\}\}.\]
On the other hand, for any $B\subseteq C^b(X)$, we have an open subset of $X$ given by
\[O_B=\bigcup_{f\in B}f^{-1}[\mathbb{C}\setminus\{0\}].\]
When $X$ is completely regular, hereditary C*-subalgebras of $C^b(X)$ distinguish open subsets of $X$ in the sense that $O=O_{B_O}$, for all $O\in\mathscr{P}(X)^\circ$.  When $X$ is compact $C^b(X)=C(X)$ and open subsets of $X$ distinguish hereditary C*-subalgebras of $C(X)$ in the sense that $B=B_{O_B}$, for all $B\in\mathcal{H}(C(X))$.  More generally, this holds when $X$ is locally compact and we consider the C*-subalgebra $C_0(X)$ of functions in $C(X)$ that vanish at infinity, i.e. those $f\in C(X)$ such that $f^{-1}[\mathbb{C}\setminus\mathbb{C}^{\epsilon\circ}]$ is compact, for all $\epsilon>0$, where $\mathbb{C}^{\epsilon\circ}=\{\lambda\in\mathbb{C}:|\lambda|<\epsilon\}$.

The Gelfand representation theorem tells us that these are, up to isomorphism, the only commutative C*-algebras.  More specifically, every commutative C*-algebra $A$ is isomorphic to $C_0(X)$ for some locally compact Hausdorff space $X$.  So in this case, by the previous paragraph, we get a natural bijection between open subsets of $X$ and hereditary C*-subalgebras of $A$.  This is why hereditary C*-subalgebras of even non-commutative C*-algebras $A$ are considered to be analogs of open subsets.\footnote{By this criterion, there are also other subsets of $A$ that could be considered as open subset analogs.  For example, we could consider closed ideals, which are precisely the hereditary C*-subalgebras when $A$ is commutative.  However, the closed ideal structure of $A$ can yield little information about (e.g. simple) non-commutative $A$.  Alternatively, we could consider more general closed bi-ideals ($B\subseteq A$ with $\overline{BAB}\subseteq B$) which, again, are precisely the hereditary C*-subalgebras when $A$ is commutative.  But this would take us outside the category of C*-algebras and into the realm of non-self-adjoint operator algebras.}

\subsection{General C*-Algebras}

Another important class of C*-algebras consists of operators on a Hilbert space.  Indeed, by the GNS construction (see \cite{Pedersen1979} \S3.3), every C*-algebra is isomorphic to a C*-subalgebra $A$ of $\mathcal{B}(H)$, the C*-algebra of all bounded linear operators on a Hilbert space $H$.  In this case, following standard practice, we denote the commutant of $A$ in $\mathcal{B}(H)$ by
\begin{equation}\label{commutant}
A'=\{b\in\mathcal{B}(H):\forall a\in A(ab=ba)\},
\end{equation}
and recall that von Neumann's double commutant theorem (see \cite{Pedersen1979} Theorem 2.2.2) says that $A''$ coincides with the weak (and strong) closure of $A$ in $\mathcal{B}(H)$.  We also define the multiplier algebra (see \cite{Pedersen1979} \S3.12) $\mathcal{M}(A)$ of $A$ by
\[\mathcal{M}(A)=\{a\in\mathcal{B}(H):aA,Aa\subseteq A\}.\]
Also important are the \emph{projections} \[\mathcal{P}(A)=\{p\in A:p^2=p=p^*\}.\]  We also define an operation $^\perp$ on projections in $\mathcal{B}(H)$ by $p^\perp=1-p$ and note that \[p\leq q\quad\Leftrightarrow\quad pq^\perp=0.\]

For $a\in A$, let $A_a$ denote the hereditary C*-subalgebra of $A$ generated by $a$, i.e. \[A_a=\overline{a^*Aa}\vee\overline{aAa^*}.\]  More generally, for $a\in A''$, let $A_a=A''_a\cap A$, so if $p\in\mathcal{P}(A'')$ then $A_p=pAp\cap A$.  On the other hand, if $B\in\mathcal{H}(A)$ then we define $p_B=\bigvee B^1_+\in\mathcal{P}(A'')$, where $A^\lambda=\{a\in A:||a||\leq\lambda\}$ denotes the closed $\lambda$-ball about $0$.  Note that $B^1_+$ is generally not a lattice however, as $B$ is a C*-algebra, $B$ has an increasing approximate unit (see \cite{Pedersen1979} Theorem 1.4.2) which has a supremum $\bigvee B^1_+$ in $A''_+$.  We call such projections \emph{open} (see \cite{Pedersen1979} Proposition 3.11.9 for some equivalent definitions), denoting them by $\mathcal{P}(A'')^\circ=\{p_B:B\in\mathcal{H}(A)\}$.  For $p\in\mathcal{B}(H)$, we naturally define the \emph{interior} $p^\circ$ of $p$ to be the largest open projection below $p$, i.e. $p^\circ=p_{A_p}$.  Note that supremums in $\mathcal{P}(A'')^\circ$ agree with supremums in $\mathcal{P}(A'')$ (see \cite{Akemann1969} Proposition II.5\footnote{Throughout \cite{Akemann1969}, $A$ is assumed to be unital and it is the universal representation of $A$ that is considered.  However, these assumptions are not necessary for this particular result.  Indeed, given $p,q\in\mathcal{P}(A'')^\circ$, we certainly have $p,q\leq(p\vee q)^\circ$ and hence $p\vee q\leq(p\vee q)^\circ\leq p\vee q$.}), but the same can not be said for infimums (see \cite{Akemann1969} Example II.6).  We also define the \emph{closure} $\overline{p}$ of $p$ by $\overline{p}=p^{\perp\circ\perp}$ and call $p$ \emph{closed} when $p=\overline{p}$, i.e. when $p^\perp$ is open.  The closed projections will be denoted by $\overline{\mathcal{P}(A'')}$ which, as a poset, is the dual of $\mathcal{P}(A'')^\circ$.

\subsection{Ideals}

Let $\mathcal{I}(A)$ denote the closed ideals of $A$.  We have $\mathcal{I}(A)\subseteq\mathcal{H}(A)$ (see \cite{Pedersen1979} Theorem 1.5.2 and Corollary 1.5.3) and the corresponding open projections are precisely those in $A'$ (see \cite{Pedersen1979} 3.11.10).  In fact, $p^\circ$ (and $\overline{p}$) lies in $A'$ whenever $p\in\mathcal{P}(A''\cap A')$, i.e.
\[\mathcal{P}(A''\cap A')^\circ=\mathcal{P}(A'')^\circ\cap A'=\{p_I:I\in\mathcal{I}(A)\}.\]
To see this, take $I\in\mathcal{I}(A)$ and note that $I''$ is then a weakly closed ideal in $A''$ with unit $p_I$.  Thus, for any $a\in A$, we have $ap_I,p_Ia\in I''$ and hence $ap_I=p_Iap_I=p_Ia$, i.e. $p_I\in A'$.  While if $p\in\mathcal{P}(A''\cap A')$ then $p^\perp\in A'$ so, for any $a,b\in A$ with $ap^\perp=0$, we have $abp^\perp=ap^\perp b=0$ and $bap^\perp=0$, i.e. $A_p=\{a\in A:ap^\perp=0\}\in\mathcal{I}(A)$ and hence $p^\circ=p_{A_p}\in\{p_I:I\in\mathcal{I}(A)\}$.

For any $a\in A''_+$, we define the \emph{central cover} $\mathrm{c}(a)$ as in \cite{Pedersen1979} 2.6.2, specifically $\mathrm{c}(a)=\bigwedge\{a'\in(A''\cap A')_+:a\leq a'\}$.  Likewise, any $B\in\mathcal{H}(A)$ has an \emph{ideal cover} $\overline{\mathrm{span}}(ABA)=\bigcap\{I\in\mathcal{I}(A):B\subseteq I\}$.  In fact, these covers correspond in the sense that, for any $B\in\mathcal{H}(A)$, we have
\[\mathrm{c}(p_B)=p_{\overline{\mathrm{span}}(ABA)}.\]  For $p_{\overline{\mathrm{span}}(ABA)}\in A''\cap A'$ and $p_B\leq p_{\overline{\mathrm{span}}(ABA)}$ so $\mathrm{c}(p_B)\leq p_{\overline{\mathrm{span}}(ABA)}$.  But also $A_{\mathrm{c}(p_B)}\in\mathcal{I}(A)$ and $B\subseteq A_{\mathrm{c}(p_B)}$ so $\overline{\mathrm{span}}(ABA)\subseteq A_{\mathrm{c}(p_B)}$ hence $p_{\overline{\mathrm{span}}(ABA)}\leq\mathrm{c}(p_B)$.

\subsection{The Reduced Atomic Representation}

From now on we assume that
\begin{center}\textbf{$A$ is a C*-algebra identified with its reduced atomic representation,}\end{center}
i.e. that $H$ is a direct sum of Hilbert spaces coming from irreducible representations (see \cite{Pedersen1979} 3.13.1) $\hat{A}$ of $A$, one for each unitary equivalence class in $\hat{A}$.  Equivalently, we assume that $A$ is identified with a C*-subalgebra of $\mathcal{B}(H)$ such that every pure state (i.e. extreme point of $A^{*1}_+$ \textendash\, see \cite{Pedersen1979} 3.10.1) on $A$ is of the form $\phi_v$ (where $\phi_v(a)=\langle av,v\rangle$, for all $a\in A$) for some unique $v\in H$.  This means that open projections distinguish hereditary C*-subalgebras in the sense that $B=A_{p_B}$, for all $B\in\mathcal{H}(A)$.  For we certainly have $B\subseteq A_{p_B}$, and if this inclusion were strict then, by \cite{Pedersen1979} Lemma 3.13.5, we would have a pure state $\phi$ on $A$ with $B\subseteq\phi^\perp\nsupseteq A_{p_B}$, where
\begin{equation}\label{phiperp}
\phi^\perp=\{a\in A:\phi(a^*a)=\phi(aa^*)=0\}
\end{equation}
which is determined by some $v\in\mathcal{R}(p_{A_{p_B}}-p_B)$, a contradiction.  Thus
\[\mathcal{P}(A'')^\circ\cong\mathcal{H}(A),\quad\textrm{via $p\mapsto A_p$ and $B\mapsto p_B$}.\]
So any order theoretic question or result about $\mathcal{H}(A)$ has an equivalent formulation in $\mathcal{P}(A'')^\circ$, and an equivalent dual formulation in $\overline{\mathcal{P}(A'')}$, and we will often find it convenient to work with $\mathcal{P}(A'')^\circ$ or $\overline{\mathcal{P}(A'')}$ instead.

\section{Cocompactness and Compactness}\label{C&C}

\begin{dfn}
$\mathbb{P}$ is \emph{compact} if $\forall S\subseteq\mathbb{P}$, $\bigvee S=1\Rightarrow\bigvee F=1$, for some finite $F\subseteq S$.
\end{dfn}

Note that $X$ is compact precisely when $\mathscr{P}(X)^\circ$ is compact by the above definition.\footnote{which is standard in point-free topology \textendash\, see \cite{PicadoPultr2012} Ch VII.  More generally, in lattice theory an element $p\in\mathbb{P}$ is said to be \emph{compact} if, $\forall S\subseteq\mathbb{P}$, $p\leq\bigvee S\Rightarrow p\leq \bigvee F=1$, for some finite $F\subseteq S$.  So $\mathbb{P}$ is a compact poset precisely when $1$ is a compact element.  However, this definition of a compact element only identifies the compact open subsets, and when it comes to open set lattices it is rather the cocompact sets we are most interested in.  Thus we give a different order theoretic definition of cocompactness in \autoref{cocomp}, and define compact projections in \autoref{compact} in the algebraic way more standard in C*-algebra theory.} And any locally compact $X$ is compact precisely when $C_0(X)$ is unital.  For this reason it is often said that unital C*-algebras are non-commutative analogs of compact topological spaces.  With the above definition we can make this more precise and identify unitality of $A$ purely from the order structure of $\mathcal{H}(A)$.

\begin{prp}\label{ucomp}
$A$ is unital precisely when $\mathcal{H}(A)$ is compact.
\end{prp}

\begin{proof}
If $A$ is unital then $1$ is q-compact, in the terminology of \cite{Akemann1971}, and hence $\mathcal{H}(A)$ is compact, by (the order theoretic dual of) \cite{Akemann1971} Theorem II.7.

Now assume $A$ is not unital.  If $A=A_a$, for some $a\in A_+$, then $0$ is a limit point of $\sigma(a)$, otherwise $a^\perp_{\{0\}}$ would be a unit in $A$.  Thus $A_{a_{(\epsilon_n,\infty)}}$ is a strictly increasing sequence in $\mathcal{H}(A)$, for some strictly decreasing $\epsilon_n\rightarrow0$, with $\bigvee A_{a_{(\epsilon_n,\infty)}}=A$, i.e. $\mathcal{H}(A)$ is not compact.  On the other hand, if $A\neq A_a$, for any $a\in A_+$, then $(A_a)_{a\in A_+}$ is an upwards directed subset of $\mathcal{H}(A)$ (as $A_{a+b}=A_a\vee A_b$ for all $a,b\in A_+$) with no maximum, even though $A=\bigvee_{a\in A_+}A_a$, i.e. $\mathcal{H}(A)$ is again not compact.
\end{proof}

\begin{dfn}\label{cocomp}
We call $p\in\mathbb{P}$ \emph{$\bigvee$-cocompact} if $\forall S\subseteq\mathbb{P}$, $p\vee\bigvee S=1\Rightarrow p\vee\bigvee F=1$, for some finite $F\subseteq S$.
\end{dfn}

Thus $\mathbb{P}$ is compact precisely when $0$ is $\bigvee$-cocompact by the above definition.  Moreover, $O\in\mathscr{P}(X)^\circ$ is $\bigvee$-cocompact precisely when $X\setminus O$ is a compact subset of $X$.  We also have a more algebraic notion of compactness.  Specifically, when $X$ is locally compact and we identify $C_0(X)''$ with $B(X)=($all arbitrary bounded functions from $X$ to $\mathbb{C})$, we see that any closed $p\in\mathcal{P}(B(X))$ is the characteristic function of a compact subset of $X$ precisely when $ap=p$, for some $a\in C_0(X)^1_+$.  This motivates the following standard definition.

\begin{dfn}\label{compact}
We call $p\in\overline{\mathcal{P}(A'')}$ \emph{compact} when $ap=p$, for some $a\in A$.
\end{dfn}

In \cite{Akemann1971} Theorem II.7 it was shown that $p$ is $\bigvee$-cocompact in $\mathcal{P}(A'')^\circ$ whenever $p^\perp$ is compact, and \cite{Akemann1971} Conjecture II.2 predicted that the converse holds (even among arbitrary regular $p\in\mathcal{P}(A'')$).  The following example refutes this conjecture.

\begin{xpl}
Take $P,Q\in\mathcal{P}(M_2)$ with $0<||PQ||<1$ and let $A$ be the C*-subalgebra of $C([0,1],M_2)$ of functions $f$ with $f(0)\in\mathbb{C}P$.  Identify $A''$ in the usual way with all bounded functions $f$ from $[0,1]$ to $M_2$ with $f(0)\in\mathbb{C}P$.  Define $q\in\mathcal{P}(A'')^\circ$ by $q(0)=0$ and $q(x)=Q$ otherwise.  As $[0,1]$ is compact and $P\neq Q$, $q$ is $\bigvee$-cocompact in $\mathcal{P}(A'')^\circ$.  But $q^\perp(x)=Q^\perp\neq P$, for all $x>0$, so $q^\perp$ is not compact.
\end{xpl}

On the other hand, in this example $r^\perp$ is compact, where $r(0)=0$ and $r(x)=P^\perp$ otherwise.  Any auto-homeomorphism $h$ of $\mathcal{P}(M_2)$ leaving $0$, $P$ and $1$ fixed gives rise to an order automorphism $\theta_h$ of $\mathcal{P}(A'')^\circ$ defined by $\theta_h(p)(x)=h(p(x))$.  If we further require that $h(P^\perp)=Q$ then $\theta_h(r)=q$ and hence $\theta_h$ does not preserve compactness (more precisely, it does not preserve the property `$p^\perp$ is compact').  Thus the compact elements of $\overline{\mathcal{P}(A'')}$ do not even admit any order theoretic characterization.

However, the injection of a little more algebra allows for a partial verification of \cite{Akemann1971} Conjecture II.2.

\begin{prp}
Any $p\in\mathcal{P}(\mathcal{M}(A))$ is $\bigvee$-cocompact in $\mathcal{P}(A'')^\circ$ iff $p^\perp$ is compact.
\end{prp}

\begin{proof}
This is esentially the same as the proof of \autoref{ucomp}.  Specifically, the `if' part follows from \cite{Akemann1971} Theorem II.7 (or alternatively one can use the correspondence $q\leftrightarrow\{a\in A^1_+:aq=q\}$ between non-zero compact projections and proper norm filters, and note that a directed union of norm centred subsets is again norm centred and hence contained in a proper norm filter \textendash\, see \cite{Bice2011}).  While if $p\in\mathcal{P}(\mathcal{M}(A))=\mathcal{P}(A'')^\circ\cap\overline{\mathcal{P}(A'')}$ and $p^\perp$ is not compact then we obtain (an increasing sequence or) upwards directed $S\subseteq\mathcal{P}(A'')^\circ$ with no maximum and $\bigvee S=p^\perp$.  Thus $p\vee\bigvee S=1$ even though $p\vee\bigvee F\neq1$ for any finite $F\subseteq S$, i.e. $p$ is not $\bigvee$-cocompact.
\end{proof}

\section{Semicomplements and Strong Orthogonality}\label{S&O}

\begin{dfn}
$p$ is a $\wedge$-semicomplement of $q$ in $\mathbb{P}$ when $p\wedge q=0$.
\end{dfn}

In $\mathscr{P}(X)^\circ$, we see that $N$ and $O$ are $\wedge$-semicomplements precisely when they are disjoint, which means that $B_OB_N=B_OC^b(X)B_N=\{0\}$.  More generally, we define the \emph{orthogonality} $\perp$ and \emph{strong orthogonality} $\triangledown$ relations on $\mathcal{H}(A)$ by
\[B\bot C\ \Leftrightarrow\ BC=\{0\}\qquad\textrm{and}\qquad B\triangledown C\ \Leftrightarrow\ BAC=\{0\}.\]
For all $B,C\in\mathcal{H}(A)$, we immediately see that
\[B\triangledown C\quad\Rightarrow\quad B\bot C\quad\Rightarrow\quad B\cap C=\{0\},\]
and these implications can not be reversed for general non-commutative $A$.  In fact, like with compactness, $\perp$ does not even admit any order theoretic characterization in $\mathcal{H}(A)$.  For example, any permutation of $\mathcal{H}(M_2)$ leaving $\{0\}$ and $M_2$ fixed is an order isomorphism, even though many of these do not preserve the $\perp$ relation.  However, we can obtain order theoretic characterizations of $\triangledown$, which is the primary goal of this section.

For this, it turns out be useful to examine the $\vee$-semicomplement (the notion dual to a $\wedge$-semicomplement) structure of $\mathcal{H}(A)$ in more detail.

\begin{dfn}
$p$ is $\vee$-\emph{separated} from $q$ in $\mathbb{P}$ if $p$ has a $\vee$-semicomplement $r$ with $q\leq r<1$.  We call $p$ \emph{subfit} if $p$ is $\vee$-separated from every $q\ngeq p$.  We call $\mathbb{P}$ itself subfit when every $p\in\mathbb{P}$ is subfit.
\end{dfn}

The term `subfit' comes from \cite{PicadoPultr2012} Ch V \S1, at least with reference to entire posets (rather than individual elements), where it is considered as an analog in point-free topology of the $T_1$ separation axiom.  To see why, we introduce atoms.

\begin{dfn}\label{atom}
An \emph{atom} of $\mathbb{P}$ is a minimal element of $\mathbb{P}\setminus\{0\}$.  We call $D\subseteq\mathbb{P}$ \emph{$\bigvee$-dense} when $p=\bigvee\{q\in D:q\leq p\}$, for all $p\in\mathbb{P}$.  We call $\mathbb{P}$ \emph{atomistic} when the atoms are $\bigvee$-dense in $\mathbb{P}$.
\end{dfn}

Dually, we define \emph{coatom}, \emph{$\bigwedge$-dense}, and \emph{coatomistic}.  If $X$ is a $T_1$ topological space then $\mathscr{P}(X)^\circ$ is coatomistic and hence subfit.  Indeed, if $\mathbb{P}$ is coatomistic and $p\nleq q$ then, as $q$ is the infimum of all coatoms above it, there must be some coatom $r\geq q$ with $r\ngeq p$ and hence $p\vee r=1$.  So $\mathcal{H}(A)$ is coatomistic and hence subfit when $A$ is commutative and, in fact, this easily generalizes to non-commutative $A$.

\begin{prp}\label{HAcoatom}
$\mathcal{H}(A)$ is coatomistic.
\end{prp}

\begin{proof}
If $B,C\in\mathcal{H}(A)$ and $B\nsubseteq C$ then there is a pure state $\phi$ on $A$ with $B\nsubseteq\phi^\perp\supseteq C$ (see \eqref{phiperp}).  As $\phi$ is pure, $\phi^\perp$ is a coatom in $\mathcal{H}(A)$, by \cite{Pedersen1979} Proposition 3.13.6.  Thus any element of $\mathcal{H}(A)$ below all coatoms greater than $C$ is below $C$, i.e. $C$ is the infimum of all such coatoms.  As $C$ was arbitrary, the coatoms are $\bigwedge$-dense in $\mathcal{H}(A)$.
\end{proof}

For another topological property related to coatoms, we introduce the following.

\begin{dfn}
We call $p\in\mathbb{P}$ \emph{$\wedge$-irreducible} if $p=q\wedge r\Rightarrow p\in\{q,r\}$, for all $q,r\in\mathbb{P}$.
\end{dfn}

Every coatom in $\mathbb{P}$ is $\wedge$-irreducible.  Also $X\setminus\overline{\{x\}}$ is $\wedge$-irreducible in $\mathscr{P}(X)^\circ$, for all $x\in X$.  We call $X$ \emph{sober} if $\mathscr{P}(X)^\circ$ has no other $\wedge$-irreducibles apart from $X$.  Among $T_0$ spaces, $T_1+$sobriety is actually a property of the lattice $\mathscr{P}(X)^\circ$,\footnote{even though neither $T_1$ nor sobriety is, individually, such lattice property (see \cite{PicadoPultr2012} I.3.1)} specifically
\[X\textrm{ is }T_1\textrm{ and sober}\ \Leftrightarrow\ X\textrm{ is $T_0$ and every $\wedge$-irreducible in $\mathscr{P}(X)^\circ$ is a coatom or }X.\]  It would be interesting to know if this also holds in $\mathcal{H}(A)$.

\begin{qst}
Is every $\wedge$-irreducible in $\mathcal{H}(A)$ a coatom?
\end{qst}

On the other hand, $\wedge$-irreducibles in $\mathcal{I}(A)$ are much more well-known.  Indeed, they are precisely the \emph{prime} $I\in\mathcal{I}(A)$, i.e. satisfying $aIb\subseteq I\Rightarrow a\in I$ or $b\in I$, for all $a,b\in A$ (see \cite{Pedersen1979} 3.13.7).  We also call $I\in\mathcal{I}(A)$ \emph{primitive} if $I$ is the largest element of $\mathcal{I}(A)$ contained in some coatom $B\in\mathcal{H}(A)$, which is equivalent saying $I$ is the kernel of some $\pi\in\hat{A}$ (see \cite{BorceuxRosickyBossche1989}).  Coatoms in $\mathcal{I}(A)$ are usually just called maximal, and we have the following relationships between these concepts in $\mathcal{I}(A)$ (see \cite{Pedersen1979} Proposition 3.13.10)
\[\textrm{maximal}\quad\Rightarrow\quad\textrm{primitive}\quad\Rightarrow\quad\textrm{prime}.\]
Note that none of these implications can be reversed in general, e.g. $\{0\}$ is primitive but not maximal in $\mathcal{B}(H)$ when $H$ is infinite dimensional, while $\{0\}$ is prime but not primitive for the $A$ constructed in \cite{Weaver2003}.

An important property of (arbitrary) $I\in\mathcal{I}(A)$ we will need is the following.

\begin{prp}\label{Idist}
For $B,C\in\mathcal{H}(A)$ and $I\in\mathcal{I}(A)$, $I\wedge(B\vee C)=(I\wedge B)\vee(I\wedge C)$.
\end{prp}

\begin{proof}
Equivalently, we need to prove that $p\wedge(q\vee r)=(p\wedge q)\vee(p\wedge r)$ in $\mathcal{P}(A'')^\circ$, whenever $p\in A'$.  But $\vee$ agrees in $\mathcal{P}(A'')^\circ$ and $\mathcal{P}(A'')$, as does $\wedge$ for commuting projections, so it suffices to verify the formula in $\mathcal{P}(A'')$.  As $q=pq+p^\perp q=pq\vee p^\perp q$ and $r=pr+p^\perp r=pr\vee p^\perp r$, we have
\[p\wedge(q\vee r)=p((pq\vee pr)+(p^\perp q\vee p^\perp r))=pq\vee pr=(p\wedge q)\vee(p\wedge r).\]
\end{proof}

To accurately describe some equivalents of the $\triangledown$ relation below, we introduce some more notation.  Firstly, let $\oplus$ denote the usual (interior) direct sum of vector spaces, so $A=B\oplus C$ means $B$ and $C$ are complementary in the lattice of (arbitary) subspaces of $A$.  Also define polarites (i.e. order reversing operations) $^\perp$ and $^\triangledown$ on $\mathscr{P}(A)$ as in \cite{Bice2014c} by
\[B^\perp=\{a\in A:\forall b\in B(ba=0=ba^*)\}\quad\textrm{and}\quad B^\triangledown=\{a\in A:\forall b\in B(bAa=\{0\})\}.\]
The elements of $\mathscr{P}(A)^\perp=\{B^\perp:B\in\mathscr{P}(A)\}$ and $\mathscr{P}(A)^\triangledown=\{B^\triangledown:B\in\mathscr{P}(A)\}$ are called \emph{*-annihilators} and \emph{annihilator ideals} respectively.

\begin{thm}\label{trieq}
For $B,C\in\mathcal{H}(A)$, the following are equivalent.
\begin{enumerate}
\item\label{BtriC} $B\triangledown C$.
\item\label{tritri} $B^{\triangledown\triangledown}\cap C^{\triangledown\triangledown}=\{0\}$.
\item\label{ideal} $\overline{\mathrm{span}}(ABA)\cap C=\{0\}$.
\item\label{semiideal} $ABA\cap C=\{0\}$.
\item\label{dsum} $B\vee C=B\oplus C$.
\item\label{semicen} $D=(B\vee D)\wedge(C\vee D)$, for all $D\in\mathcal{H}(A)$.
\item\label{del} $B\wedge D=B\wedge(C\vee D)$, for all $D\in\mathcal{H}(A)$.
\item\label{vsemi} Every $\vee$-semicomplement of $C$ in $\mathcal{H}(A)$ contains $B$.
\item\label{mirr} Every $\wedge$-irreducible in $\mathcal{H}(A)$ contains $B$ or $C$.
\item\label{coatom} Every coatom in $\mathcal{H}(A)$ contains $B$ or $C$.
\item\label{prime} Every prime $I\in\mathcal{I}(A)$ contains $B$ or $C$.
\item\label{prim} Every primitive $I\in\mathcal{I}(A)$ contains $B$ or $C$.
\item\label{aa*} $aa^*\in B$ and $a^*a\in C\Rightarrow a=0$, for all $a\in A$.
\end{enumerate}
\end{thm}

\begin{proof} We immediately see that \eqref{tritri}$\Rightarrow$\eqref{ideal}$\Rightarrow$\eqref{semiideal}, \eqref{semicen}$\Rightarrow$\eqref{del}$\Rightarrow$\eqref{vsemi}$\Rightarrow$\eqref{coatom}, \eqref{semicen}$\Rightarrow$\eqref{mirr}$\Rightarrow$\eqref{coatom} and \eqref{prime}$\Rightarrow$\eqref{prim}.  The rest of the equivalences are proved as follows.

\begin{itemize}

\item[\eqref{BtriC}$\Rightarrow$\eqref{tritri}] As $B\triangledown C$, we have $B\subseteq C^\triangledown$ and hence $C^{\triangledown\triangledown}\subseteq B^\triangledown$ which, as $B^\triangledown\cap B^{\triangledown\triangledown}=\{0\}$, means $B^{\triangledown\triangledown}\cap C^{\triangledown\triangledown}=\{0\}$.
\item[\eqref{semiideal}$\Rightarrow$\eqref{BtriC}] Take $a\in BAC$.  Then $a^*a\in ABA\cap C=\{0\}$ so $a=0$.

\item[\eqref{ideal}$\Rightarrow$\eqref{vsemi}] Let $I=\overline{\mathrm{span}}(ABA)$.  If $D$ is a $\vee$-semicomplement of $C$ in $\mathcal{H}(A)$ then, by \autoref{Idist}, $I=I\wedge A=I\wedge(C\vee D)=(I\wedge C)\vee(I\wedge D)=I\wedge D$, so $B\subseteq I\subseteq D$.
\item[\eqref{vsemi}$\Rightarrow$\eqref{semicen}] As $\mathcal{H}(A)$ is subfit, by \autoref{HAcoatom}, $\mathcal{H}(A)$ is SSC*, in the terminology of \cite{MaedaMaeda1970}, where this implication appears as Theorem (4.18) $(\beta)\Rightarrow(\gamma)$.
\item[\eqref{coatom}$\Rightarrow$\eqref{prim}] Say we had some primitive ideal which contained neither $B$ nor $C$, i.e. we have $\pi\in\hat{A}$ with $\pi[B]\neq\{0\}\neq\pi[C]$.  Thus we have some $b\in B$, $c\in C$ and $v\in H_\pi$ with $\pi(b)v\neq0\neq\pi(c)v$.  Indeed, if $\pi(p_B)\pi(p_C)\neq0$ then we can pick $v\in\mathcal{R}(\pi(p_B))$ (or $v\in\mathcal{R}(\pi(p_C))$), while if $\pi(p_B)\pi(p_C)=0$ then we can set $v=x+y$ for any $x\in\mathcal{R}(\pi(p_B))\setminus\{0\}$ and $y\in\mathcal{R}(\pi(p_B))\setminus\{0\}$.  As $\pi$ is irreducible, $\phi_v=\langle\pi(\cdot)v,v\rangle$ is a pure state, and hence $\phi_v^\perp$ is a coatom containing neither $B$ nor $C$.
\item[\eqref{prim}$\Rightarrow$\eqref{aa*}] By \eqref{prim}, $\pi(aa^*)=0$ or $\pi(a^*a)=0$, and hence $\pi(a)=0$, for every $\pi\in\hat{A}$.  As $\pi$ was arbitrary, $a=0$.
\item[\eqref{aa*}$\Rightarrow$\eqref{BtriC}] If $a\in BAC$ then $aa^*\in BAB\subseteq B$ and $a^*a\in CAC\subseteq C$ so $a=0$.

\item[\eqref{BtriC}$\Rightarrow$\eqref{dsum}] As $B$ and $C$ are C*-subalgebras of $A$ with $B\perp C$, $B\oplus C$ is also a C*-subalgebra of $A$.  As $B$ and $C$ are hereditary and $BAC=\{0\}$, we also have $(B+C)A(B+C)\subseteq BAB+CAC\subseteq B+C$.  Thus $B\oplus C$ is also hereditary, by \cite{Blackadar2013} Corollary II.5.3.9.
\item[\eqref{dsum}$\Rightarrow$\eqref{aa*}] Take $a\in A$ with $aa^*\in B$ and $a^*a\in C$, so $a\in B\vee C=B\oplus C$, i.e. $a=b+c$, for some $b\in B$ and $c\in C$.  Hence \[aa^*-ab^*-ba^*+bb^*=(a-b)(a^*-b^*)=cc^*\in B\cap C=\{0\}.\]  Thus $c=0$ and, likewise, $b=0$ and hence $a=b+c=0$.

\item[\eqref{BtriC}$\Rightarrow$\eqref{prime}] Assume $B\triangledown C$, which we already know is equivalent to $B^{\triangledown\triangledown}\triangledown C^{\triangledown\triangledown}$.  Say $I\in\mathcal{I}(A)$ contains neither $B$ nor $C$, so $I\subsetneqq I\vee B^{\triangledown\triangledown},I\vee C^{\triangledown\triangledown}$, even though $I=(I\vee B^{\triangledown\triangledown})\wedge(I\vee C^{\triangledown\triangledown})$, by \eqref{semicen}, i.e. $I$ is not $\wedge$-irreducible in $\mathcal{I}(A)$.
\end{itemize}
\end{proof}

Thus \eqref{semicen}, \eqref{del}, \eqref{vsemi}, \eqref{mirr} and \eqref{coatom} give us purely order theoretic characterizations of $\triangledown$ (and in lattice theory \eqref{del} is often also denoted by $\triangledown$ and called the `del' relation).  The same could be said of \eqref{tritri}, \eqref{ideal}, \eqref{prime} and \eqref{prim} once it is known that ideals and annihilator ideals have order theoretic characterizations, as shown in \S\S\ref{*AI} and \ref{D&I}.  Alternatively, note that $B^\triangledown$ is the maximum $C$ in $\mathcal{H}(A)$ or $\mathcal{I}(A)$ with $B\triangledown C$, so the order-theoretic characterizations of $\triangledown$ here yield order theoretic characterizations of annihilator ideals.  Also, it would be interesting to know if $B\triangledown C$ even when $B$ and $C$ are complementary in the lattice of \emph{closed} subspaces of $B\vee C$, i.e. whether \eqref{dsum} can be weakened to $B\vee C=\overline{B+C}$ and $B\cap C=\{0\}$.

Incidentally, one might also consider the algebraic relation $A=B+C$ to be something of a dual to $\triangledown$.  Indeed, it agrees with the $\vee$-semicomplement relation $A=B\vee C$ when $A$ is commutative (as then $B$ and $C$ are ideals so $B\vee C=B+C$ \textendash\, see \cite{Pedersen1979} Corollary 1.5.8), but is significantly stronger for non-commutative $A$.

\begin{qst}
Can some dual to \autoref{trieq} be proved for the relation $A=B+C$?
\end{qst}

\section{Complements and Projections}\label{C&P}

\begin{dfn}
We call $p,q\in\mathbb{P}$ \emph{complementary} when \[p\wedge q=0\quad\textrm{and}\quad p\vee q=1.\]
\end{dfn}

The complements in $\mathscr{P}(X)^\circ$ are precisely the clopen (i.e. closed and open) subsets.  Also, the projections in $C^b(X)(\cong\mathcal{M}(C_0(X))$ when $X$ is locally compact$)$ are precisely the characteristic functions of clopen subsets of $X$.  Thus complements in $\mathcal{H}(A)$ correspond to projections in $\mathcal{M}(A)$ whenever $A$ is commutative.  The motivating question for this section is whether this extends to non-commutative $A$.  Phrased in terms of $\mathscr{P}(A'')^\circ\cong\mathcal{H}(A)$, the following result immediately provides a partial answer.

\begin{prp}\label{pcomp}
If $p\in\mathcal{P}(\mathcal{M}(A))$ then $p$ and $p^\perp$ are complementary in $\mathcal{P}(A'')^\circ$.
\end{prp}

\begin{proof}
As $\mathcal{P}(\mathcal{M}(A))=\{p\in\mathcal{P}(A''):p^\circ=\overline{p}\}$, by \cite{Pedersen1979} Theorem 3.12.9, $p^\perp\in\mathcal{P}(A'')^\circ$.  But $p$ and $p^\perp$ are complementary in $\mathcal{P}(A'')$ and so certainly in $\mathcal{P}(A'')^\circ$.
\end{proof}

We can also prove the converse, but only when one of the projections is compact.  The proof also requires the following elementary results.

Denote the spectral projection in $A''$ of $a\in A_+$ corresponding to $S\subseteq\mathbb{R}_+$ by $a_S$.  For any $a,b\in A_+$, $\epsilon>0$, and $v\in\mathcal{R}((a+b)_{[0,\epsilon^3]})$, \[\epsilon||a_{(\epsilon,\infty)}v||^2=\epsilon\langle a_{(\epsilon,\infty)}v,v\rangle\leq\langle av,v\rangle\leq\langle(a+b)v,v\rangle\leq\epsilon^3\langle v,v\rangle=\epsilon^3||v||^2,\textrm{ so}\]
\begin{equation}\label{a+beq}
||(a+b)_{[0,\epsilon^3]}a_{(\epsilon,\infty)}||\leq\epsilon.
\end{equation}

Also, for any $p_1,\ldots,p_n\in\mathcal{P}(A)$, $\epsilon>0$ and unit $v\in H$ such that $||p_k^\perp v||\leq\epsilon$, for all $k\leq n$, we have,
\begin{equation}\label{p1pneq}
||p_1\ldots p_n||\geq||p_1\ldots p_nv||\geq||p_1\ldots p_{n-1}v||-\epsilon\geq\ldots\geq1-n\epsilon.
\end{equation}

\begin{thm}\label{unitcomp}
If $p,q\in\mathcal{P}(A'')^\circ$ are complementary and $q^\perp$ is compact then $p\in A$ and $||pq||=||p^\perp q^\perp||<1$.
\end{thm}

\begin{proof}
Assume, to the contrary, that $||pq||=1$.  As $\sigma(pq)=\sigma(pqp)$, by \cite{HladnikOmladic1988}, if $1>\sup(\sigma(pq)\setminus\{1\})$ then $0\neq p\wedge q\in\mathcal{P}(A'')^\circ$, by \cite{Akemann1969} Theorem II.7\footnote{Actually, there is slight oversight in the proof of \cite{Akemann1969} Theorem II.7.  Specifically, in paragraph 2 line 2, $q$ is replaced with $q_0=q-p\wedge q$, which is fine until line 4 from the bottom, where we must revert back to the original $q$.}, contradicting the assumption that $(p\wedge q)^\circ=0$.  But $\sigma(pq)\setminus\{0,1\}=\sigma(p^\perp q^\perp)\setminus\{0,1\}$, by \cite{Bice2012} \S2.2, so if $1=\sup(\sigma(pq)\setminus\{1\})=\sup(\sigma(p^\perp q^\perp)\setminus\{1\})$ then $||p^\perp q^\perp||=1$.  And if $||p^\perp q^\perp||=1$ then we have a state $\phi$ on $A''$ with $\phi(p^\perp)=1=\phi(q^\perp)$, by \cite{Bice2011} Theorem 2.2.  As $q^\perp$ is compact, $q^\perp\leq a$ and hence $\phi(a)=1$, for some $a\in A^1_+$.  Thus $\phi$ restricts to a state on $A$, so defining $\phi^\perp$ as in \eqref{phiperp} yields $p,q\leq p_{\phi^\perp}<1$ (note $p$ is open so $0<p=\bigvee A^1_{p+}$ and hence $\phi(a)=0$ for some $a\in A_+\setminus\{0\}$, i.e. $\phi$ is not faithful on $A$), contradicting the assumption that $p\vee q=1$.  Thus $||pq||=\sqrt{\sup(\sigma(pq))}=\sqrt{\sup(\sigma(p^\perp q^\perp))}=||p^\perp q^\perp||<1$.

Now suppose that $p\notin A$, so $b_{(\epsilon,\infty)}<p$, for all $b\in A_{p+}$ and $\epsilon>0$.  We claim that, moreover, $((q\vee b_{(\epsilon,\infty)})^\perp)_{b\in A_{p+},\epsilon>0}$ is norm centred (see \cite{Bice2011} Definition 2.1).  To see this, take $b_1,\ldots,b_n\in A_{p+}$ and $\epsilon_1,\ldots,\epsilon_n>0$.  Let $b=b_1+\ldots+b_n\in A_{p+}$ and, for any $\epsilon>0$, let $\delta=\min(\epsilon^3,\epsilon_1^3,\ldots,\epsilon_n^3)$.  By \eqref{a+beq}, \[||b_{[0,\delta]}(b_k)_{(\epsilon_k,\infty)}||\leq\epsilon,\] for all $k\leq n$.  Thus, by the inequality in line 4 of the proof \cite{Bice2012} Lemma 2.7 (where $P=q$, $Q=(b_k)_{(\epsilon_k,\infty)}$ and $R=b_{(\delta,\infty)}$), \[||(q\vee b_{(\delta,\infty)})^\perp(q\vee (b_k)_{(\epsilon_k,\infty)})||\leq\epsilon/\sqrt{1-||pq||^2}.\]  As $||pq||<1$, we have $\mathcal{R}(r\vee q)=\mathcal{R}(r)+\mathcal{R}(q)$, for any $r\leq p$, i.e. the supremum of these projections is just the supremum of the corresponding subspaces.  As we also know that subspaces of a vector space are modular, \[(b_{(\delta,\infty)}\vee q)\wedge p=b_{(\delta,\infty)}\vee(q\wedge p)=b_{(\delta,\infty)}<p\] so $b_{(\delta,\infty)}\vee q\neq1$.  Thus, we can take unit $v\in\mathcal{R}(q\vee b_{(\delta,\infty)})^\perp$ and \eqref{p1pneq} yields
\[||(q\vee(b_1)_{(\epsilon_1,\infty)})^\perp...(q\vee(b_n)_{(\epsilon_n,\infty)})^\perp||\geq 1-n\epsilon/\sqrt{1-||pq||^2}.\]
As $\epsilon$ was arbitrary, we in fact have $||(q\vee(b_1)_{(\epsilon_1,\infty)})^\perp...(q\vee(b_n)_{(\epsilon_n,\infty)})^\perp||=1$ which, as $b_1,\ldots,b_n$ and $\epsilon_1,\ldots,\epsilon_n$ were arbitrary, shows that $((q\vee b_{(\epsilon,\infty)})^\perp)_{b\in A_{p+},\epsilon>0}$ is indeed norm centred.  Thus we have a state $\phi$ on $A''$ with $\phi((q\vee b_{(\epsilon,\infty)})^\perp)=1$, for all $b\in A_{p+}$ and $\epsilon>0$, which means that $\phi[A_p]=\{0\}=\phi[A_q]$.  Again, as $q^\perp$ is compact, $\phi$ restricts to a state on $A$ so $p,q\leq p_{\phi^\perp}<1$, contradicting $p\vee q=1$.
\end{proof}

In particular, if $A$ is unital then any complementary $p,q\in\mathcal{P}(A'')^\circ$ must lie in $A$ and $\mathcal{P}(A)$ is a (first order) definable subset of the lattice $\mathcal{P}(A'')^\circ$ (the weaker statement that $\mathcal{P}(A)$ can be determined from the (non-first order) lattice structure of $\mathcal{P}(A'')^\circ$ follows already from \autoref{ucomp}).  One might conjecture that even when $A$ is non-unital, complementary $p,q\in\mathcal{P}(A'')^\circ$ must lie in $\mathcal{M}(A)$.  The following example shows this to be false.

\begin{xpl}\label{nonunital}
Take $P,Q\in\mathcal{P}(M_2)$ with $0<||PQ||<1$ and let $A$ be the C*-subalgebra of $C([0,1],M_2)$ of functions $f$ with $f(0)\in\mathbb{C}P$ and $f(1)\in\mathbb{C}Q$.  Identify $A''$ in the usual way with all bounded functions $f$ from $[0,1]$ to $M_2$ with $f(0)\in\mathbb{C}P$ and $f(1)\in\mathbb{C}Q$.  Define $p,q\in\mathcal{P}(A'')^\circ$ by $p(1)=0$, $p(x)=P$ otherwise, $q(0)=0$ and $q(x)=Q$ otherwise.  Then $p$ and $q$ are complementary (not just in $\mathcal{P}(A'')^\circ$ but even in $\mathcal{P}(A'')$) even though neither $p$ nor $q$ is in $\mathcal{M}(A)$.
\end{xpl}

\begin{qst}
Is there any algebraic characterization (or anything non-trivial that can be said) of complements in $\mathcal{H}(A)$ when $A$ is not unital?
\end{qst}

Strengthening the hypothesis in \autoref{pcomp} also yields uniqueness.

\begin{prp}\label{punicomp}
If $p\in\mathcal{P}(\mathcal{M}(A)\cap A')$, $p^\perp$ is the unique complement of $p$ in $\mathcal{P}(A'')^\circ$.
\end{prp}

\begin{proof}
In $\mathcal{P}(A'')^\circ$, we have $q\wedge r=qr$, whenever $qr=rq$, and $q\vee r=q+r$, whenever $qr=0$.  Thus $p\wedge p^\perp=pp^\perp=0$ and $p\vee p^\perp=p+p^\perp=1$, i.e. $p^\perp$ is a complement of $p$.  Moreover, if $q$ is another complement of $p$ then $pq=p\wedge q=0$ so $p+q=p\vee q=1$, i.e. $q=p^\perp$.
\end{proof}

Again, we can prove the converse for compact projections.

\begin{cor}\label{unitunicomp}
If compact $p$ is a unique complement of $q\in\overline{\mathcal{P}(A'')}$ then $p\in A\cap A'$.
\end{cor}

\begin{proof}
By \autoref{unitcomp}, $q^\perp\in A$ so $q\in\mathcal{M}(A)$.  By \autoref{pcomp}, $q^\perp$ is a complement or $q$ which, if $p$ \emph{is} the unique complement of $q$, gives $p=q^\perp\in A$.  Assuming $p\notin A'$, we would have $a\in A^1_+$ with $ap\neq pa$.  As $||a||\leq1<\pi$, $a=f(e^{ia})$, where $f$ is a continuous function on $\mathbb{T}=\mathbb{C}^1\setminus\mathbb{C}^{1\circ}$, so $p$ does not commute with $e^{ia}$ either, i.e. $e^{ia}pe^{-ia}\neq p$.  Multiplying $a$ by some $\epsilon>0$ if necessary, we can also make $||1-e^{ia}||<\frac{1}{2}$ so $||p-e^{ia}pe^{-ia}||<1$, which means $||qe^{ia}pe^{-ia}||=||q^\perp(e^{ia}pe^{-ia})^\perp||<1$.  Thus $e^{ia}pe^{-ia}$ is a complement of $q$ in $\mathcal{P}(A'')$ and so certainly in $\overline{\mathcal{P}(A'')}$, i.e. $q$ has more than one complement, a contradiction (see also \cite{Ozawa2014}).
\end{proof}

So if $A$ is unital and $p$ is or has a unique complement in $\mathcal{P}(A'')^\circ$ then $p\in A$.  Unlike \autoref{unitcomp}, this might lead to a non-compact version of \autoref{unitunicomp}.

\begin{qst}
Is $p\in\mathcal{M}(A)\cap A'$ whenever $p$ is/has a unique complement in $\mathcal{P}(A'')^\circ$?
\end{qst}

\section{Pseudocomplements and Annihilator Ideals}\label{*AI}

\begin{dfn}[\cite{Birkhoff1967} Ch V \S8]
$p$ is a \emph{$\wedge$-pseudocomplement} of $q$ in $\mathbb{P}$ when
\[p=\bigvee\{r:q\wedge r=0\}\quad\textrm{and}\quad p\wedge q=0.\]
\end{dfn}

In other words, a $\wedge$-pseudcomplement is a maximum $\wedge$-semicomplement.  Every $O\in\mathscr{P}(X)^\circ$ has a $\wedge$-pseudocomplement given by $(X\setminus O)^\circ$.  So $\wedge$-pseudocomplements in $\mathscr{P}(X)^\circ$ are precisely the interiors of closed sets, i.e. the \emph{regular} open sets.  Also,
\[\mathscr{P}(C_0(X))^\perp=\mathscr{P}(C_0(X))^\triangledown=\{B_O:O\textrm{ is regular open}\},\]
for any locally compact $X$.  Thus, when $A$ is commutative, we have natural bijections between *-annihilators/annihilator ideals and $\wedge$-pseudocomplements in $\mathcal{H}(A)$.  A quick check of elementary non-commutative $A$ yields *-annihilators that are not $\wedge$-pseudocomplements, which makes it only reasonable to conjecture that $\wedge$-pseudocomplements in $\mathcal{H}(A)$ are precisely the annihilator ideals.  This is the central result we prove in this section.

\begin{prp}\label{halfa}
If $I\in\mathcal{H}(A)$ is an ideal, $I^\perp$ is a $\wedge$-pseudocomplement of $I$ and
\[I\textrm{ is a $\wedge$-pseudocomplement}\quad\Leftrightarrow\quad I=I^{\perp\perp}.\]
\end{prp}

\begin{proof}
Say $B\in\mathcal{H}(A)$ and $B\cap I=\{0\}$.  If $B\nsubseteq I^\perp$, then there exists $b\in B$ and $a\in I$ such that $ab\neq0$ and hence $0\neq b^*a^*ab\in B\cap I$, as $B$ is hereditary and $I$ is an ideal, a contradiction.  As $B$ was arbitrary, $I^\perp$ is a $\wedge$-pseudocomplement of $I$.  In particular, as $I^\perp$ is also an ideal, $I^{\perp\perp}$ is a $\wedge$-pseudocomplement of $I^\perp$.  And if $I$ is a $\wedge$-pseudocomplement of some $C\in\mathcal{H}(A)$ then $C\subseteq I^\perp$.  For, if not, we could find $a\in I$ and $c\in C$ such that $ac\neq0$ and hence $0\neq c^*a^*ac\in I\cap C$, as $I$ is an ideal and $C$ is hereditary, contradicting $I\cap C=\{0\}$.  Thus we have $I\subseteq I^{\perp\perp}\subseteq C^\perp\subseteq I$, the last inclusion coming from the defining property of a $\wedge$-pseudocomplement (and the fact $C\cap C^\perp=\{0\}$), i.e. $I=I^{\perp\perp}$.
\end{proof}

Note that $B^\perp=B^\triangledown$ whenever $B$ is a (right) ideal, so the above result implies that every annihilator ideal is a $\wedge$-pseudocomplement, and we now set about proving the converse.  This requires the following lemma, which will come in handy again later on.  It gives a simple algebraic description of the restriction of a certain Sasaki projection (see \cite{Kalmbach1983} \S5 p99) on $\mathcal{H}(A_a)$.  Firstly, let $u=u_a\in A''$ denote the partial isometry coming from the polar decomposition of $a\in A$ (see \cite{Pedersen1979} Proposition 2.2.9), so $u^*u=(a^*a)_{\{0\}}^\perp$, $uu^*=(aa^*)_{\{0\}}^\perp$ and $a=u|a|=|a^*|u$.

\begin{lem}\label{Sasaki}
If $a\in A$ and $a^2=0$ then $u=u_a\in\mathcal{M}(A_a)$ and, for any $B\in\mathcal{H}(vA_av^*)$, where $v=v_a=\frac{1}{\sqrt{2}}(u+u^*u)$, we have
\begin{equation}\label{Sasakieq}
\overline{a^*Aa}\cap(B\vee\overline{aAa^*})=v^*Bv.
\end{equation}
\end{lem}

\begin{proof}
For continuous $f$ on $\mathbb{R}_+$ with $f(0)=0$, we have $uf(a^*a),f(aa^*)u\in A_a$ (see \cite{Cuntz1977} Proposition 1.3).  For such $(f_n)\uparrow\chi_{(0,\infty)}$ (pointwise) and for any $d\in A_a$, we have $(f_n(a^*a)+f_n(aa^*))d\rightarrow d$ and $d(f_n(a^*a)+f_n(aa^*))\rightarrow d$ so \[ud=\lim uf_n(a^*a)d\in A_a\quad\textrm{and}\quad du=\lim df_n(a^*a)u\in A_a.\]  As $d\in A_a$ was arbitrary, $u\in\mathcal{M}(A_a)$.

Take $p\leq vv^*$.  As $\sqrt{2}(1-uu^*)vv^*=\frac{1}{\sqrt{2}}(u^*+u^*u)=v^*$, we have $\sqrt{2}(1-uu^*)p=v^*p$ and hence $v^*pv=2(1-uu^*)p(1-uu^*)\leq p\vee uu^*$ so $uu^*+v^*pv\leq p\vee uu^*$.  Also $v^*uu^*v=\frac{1}{2}u^*u=\frac{1}{2}v^*v$ and hence $vv^*uu^*vv^*=\frac{1}{2}vv^*$ so $puu^*p=\frac{1}{2}p$.  But also $v^2=\frac{1}{2}(u+u^*u)=\frac{1}{\sqrt{2}}v$ and hence $vv^*vvv^*=\frac{1}{\sqrt{2}}vv^*$ so $pvp=\frac{1}{\sqrt{2}}p$ and therefore $pv^*pvp=\frac{1}{2}p$.  Thus $p(uu^*+v^*pv)p=p$, i.e. $p\leq uu^*+v^*pv$.  We certainly also have $uu^*\leq uu^*+v^*pv$ and thus $p\vee uu^*=uu^*+v^*pv$.  Hence
\[u^*u\wedge(p\vee uu^*)=v^*pv.\]
If $p$ is also open then so is $v^*pv=\bigvee v^*A^1_{p_+}v$, thus verifying \eqref{Sasakieq}.
\end{proof}

Incidentally, it would be interesting to know if $u_a\in\mathcal{M}(A_a)$ even when $a^2\neq0$.  Also, an important special case of \autoref{Sasaki} occurs when $B=vA_av^*$, in which case \eqref{Sasakieq} becomes $\overline{a^*Aa}\cap(vA_av^*\vee\overline{aAa^*})=\overline{a^*Aa}$, i.e.
\begin{equation}\label{Sasakieq2}
\overline{a^*Aa}\subseteq vA_av^*\vee\overline{aAa^*}
\end{equation}
Note too that $vv^*=\frac{1}{2}(uu^*+u+u^*+u^*u)$ so, as $u_{a^*}=u_a^*$, we have $v_av_a^*=v_{a^*}v_{a^*}^*$.  Thus, replacing $a$ with $a^*$ in \eqref{Sasakieq} we get $\overline{a^*Aa}\cap(B\vee\overline{aAa^*})=v^*_{a^*}Bv_{a^*}$ and \eqref{Sasakieq2} becomes
\begin{equation}\label{Sasakieq3}
\overline{aAa^*}\subseteq vA_av^*\vee\overline{a^*Aa}
\end{equation}

\begin{thm}\label{pseudo}
If $B$ is a $\wedge$-pseudocomplement of $C$ in $\mathcal{H}(A)$ then $B=C^\perp=C^\triangledown$.
\end{thm}

\begin{proof}
As $C\cap C^\perp=\{0\}$, we have $C^\triangledown\subseteq C^\perp\subseteq B$, by the definition of $\wedge$-pseudocomplement.  Thus it suffices to prove $B\subseteq C^\triangledown$ or, equivalently, $B\triangledown C$.

Assume to the contrary that there is some $a\in A\setminus\{0\}$ with $a^*a\in B$ and $aa^*\in C$.  We first claim that, by finding a suitable replacement if necessary, we can further assume that $a^2=0$.  To see this, take $\delta\in(0,||a||)$ and let $f$ and $g$ be continuous functions on $\mathbb{R}_+$ with $f(0)=0$, $f(r)=1$, for all $r\geq\delta$, $g(r)=0$, for all $r\leq\delta$, and $g(r)\in(0,1)$, for all $r>\delta$ (so $fg=g$).  For $u=u_a$, again by \cite{Cuntz1977} Proposition 1.3 we have $d=ug(|a|)=g(|a^*|)u\in A$.  Let $e=d-f(|a|)d$ and note that, as $g(|a|)f(|a|)=g(|a|)$, we have $df(|a|)=d$ so $de=0$ and hence $e^2=0$.  Thus if there exists $c\in\overline{eAe^*}\cap C\setminus\{0\}$ then we may replace $a$ with $ce$ and we are done.  Otherwise, $\overline{eAe^*}\cap C=\{0\}$ and hence $\overline{eAe^*}\subseteq B$, as $B$ is a $\wedge$-pseudocomplement of $C$.  Note that
\begin{eqnarray*}
ee^* &=& (1-f(|a|))dd^*(1-f(|a|))\\
&=& dd^*-f(|a|)dd^*-dd^*f(|a|)+f(|a|)dd^*f(|a|)\\
&\geq& dd^*-(dd^*)^2-f(|a|)^2+f(|a|)dd^*f(|a|),
\end{eqnarray*}
(we are using the fact that $x^2-xy-yx+y^2=(x-y)^2\geq0$, for any $x,y\in A_\mathrm{sa}$, where here $x=dd^*$ and $y=f(|a|)$) so, as $ee^*,a^*a\in B$,
\begin{equation}\label{dd*}
dd^*-(dd^*)^2\leq ee^*+f(|a|)^2-f(|a|)dd^*f(|a|)\in B.
\end{equation}
But $dd^*=g(|a^*|)^2$ so $||d||^2<1$ and hence $0<dd^*-(dd^*)^2\in C$, as $aa^*\in C$, contradicting $B\cap C=\{0\}$.  Thus the claim is proved.

So, further assuming that $a^2=0$, we may consider $D=vA_av^*\in\mathcal{H}(A)$ as in \autoref{Sasaki}.  If $C\cap D=\{0\}$ then, as $B$ is a $\wedge$-pseudocomplement of $C$, we have $D\subseteq B$.  But then, by \eqref{Sasakieq3},
\[\overline{aAa^*}\subseteq D\vee\overline{a^*Aa}\subseteq B,\] even though $aa^*\in C$, contradicting $B\cap C=\{0\}$.  But if $E=C\cap D\neq\{0\}$ then $E\vee\overline{aAa^*}\subseteq C$ even though
\[\{0\}\neq v^*Ev=\overline{a^*Aa}\cap(E\vee\overline{aAa^*})\subseteq B,\] again contradicting $B\cap C=\{0\}$.
\end{proof}

Even when $B$ has no $\wedge$-pseudocomplement $\mathcal{H}(A)$, we can still prove something about the supremum of all $\wedge$-semicomplements of $B$ (see \autoref{supsemi}) using the following result, which strengthens \cite{Pedersen1979} Lemma 2.6.3. 

Let $\mathcal{U}(S)$ denote the unitaries in $S\subseteq\mathcal{M}(A)$, and let $\mathrm{c}(a)$ denote the central cover of $a$ in $A''_\mathrm{sa}$, as in \cite{Pedersen1979} 2.6.2.  So if $p\in\mathcal{P}(A'')$ then $\mathrm{c}(p)$ is the smallest element of $\mathcal{P}(A''\cap A')$ with $p\leq\mathrm{c}(p)$ and, moreover, $\mathrm{c}(p)$ is open whenever $p$ is, in which case $\mathrm{c}(p)=\bigvee(\overline{\mathrm{span}}(AA_pA))^1_+$.

\begin{lem}\label{eplem}
For every $\epsilon>0$ and $p\in\mathcal{P}(A'')$, we have $\mathrm{c}(p)=\bigvee_{u\in\mathcal{U}(1+A^\epsilon)}upu^*$.
\end{lem}

\begin{proof}
If $p_U=\bigvee_{u\in\mathcal{U}(1+A^\epsilon)}upu^*<\mathrm{c}(p)$ then we have an irreducible representation $\pi$ of $A$ such that $\pi(p)\neq0$ and $\pi(p_U)\neq 1$.  Therefore there exist unit vectors $v\in\mathcal{R}(\pi(p))\subseteq\mathcal{R}(\pi(p_U))$ and $w\in\mathcal{R}(\pi(p_U))^\perp$.  Let $b$ be a self-adjoint operator of norm $\pi/2$ on $\mathrm{span}(v,w)$ such that $e^{ib}(v)=w$ and $e^{ib}(w)=v$.  By Kadison's transitivity theorem we have $c\in A_\mathrm{sa}$ such that $\pi(c)$ agrees with $b$ on $\mathrm{span}(v,w)$ and hence, for all $t\in(0,1)$, $e^{itc}pe^{-itc}\nleq p_U$ (because $w\in\mathcal{R}(p\vee e^{itc}pe^{-itc})$ and $p\leq p_U$), contradicting the definition of $p_U$.
\end{proof}

\begin{prp}\label{supsemi}
For any $p\in\mathcal{P}(A'')^\circ$, we have $\mathrm{c}(p^{\perp\circ})\leq\bigvee\{q\in\mathcal{P}(A'')^\circ:p\wedge q=0\}$.
\end{prp}

\begin{proof}
Note that $pp^{\perp\circ}=0$ so, for all $u\in\mathcal{U}(1+A^{1/3})$, $||pup^{\perp\circ}u^*||<1$, and hence $p\wedge up^{\perp\circ}u^*=0$.  Thus $\mathrm{c}(p^{\perp\circ})\leq r=\bigvee\{q\in\mathcal{P}(A'')^\circ:p\wedge q=0\}$, by \autoref{eplem}.
\end{proof}

To see that the inequality here can be strict, see \autoref{C01K} below.

This also gives us a simple way of proving a weakened form of \autoref{pseudo}.

\begin{cor}
If $p$ has a $\wedge$-pseudocomplement in $\mathscr{P}(A'')^\circ$ then $p^{\perp\circ}\in A'$.
\end{cor}

\begin{proof}
By \autoref{supsemi}, $\mathrm{c}(p^{\perp\circ})\leq q$, where $q$ is the $\wedge$-pseudocomplement of $p$, so $p\mathrm{c}(p^{\perp\circ})=(p\wedge\mathrm{c}(p^{\perp\circ}))^\circ=0$ and hence $\mathrm{c}(p^{\perp\circ})=p^{\perp\circ}$, i.e. $p^{\perp\circ}\in A'$.
\end{proof}

\section{Distributivity and Ideals}\label{D&I}

\begin{dfn}\label{veedistdef}
In a lattice $\mathbb{P}$ we call $p\in\mathbb{P}$ \emph{$\vee$-distributive} if, for all $q,r\in\mathbb{P}$, \[p\wedge(q\vee r)=(p\wedge q)\vee(p\wedge r).\]
Likewise, we call $p\in\mathbb{P}$ \emph{$\bigvee$-distributive} if, for all $S\subseteq\mathbb{P}$, \[p\wedge(\bigvee S)=\bigvee_{s\in S}(p\wedge s).\]
\end{dfn}

Every $O\in\mathscr{P}(X)^\circ$ is $\bigvee$-distributive\footnote{Indeed, most classical point-set topology can be done, albeit with various subtle differences, in the context of \emph{frames/locales} which are, by definition, just complete (bounded) lattices in which every element is $\bigvee$-distributive (see \cite{PicadoPultr2012}).} and, more generally, it is well-known in quantale theory (see \cite{BorceuxBossche1986}) that ideals in $\mathcal{H}(A)$ are $\bigvee$-distributive (alternatively this can be proved using open projections, as in \autoref{Idist}).  Conversely it was shown in \cite{BorceuxRosickyBossche1989} Proposition 5 that, for liminal $A$, certain $\vee$-distributive elements are ideals, where it was also asked what the general $\vee$-distributive elements in (even non-liminary) $A$ look like.  It turns out that the obvious candidate works.

\begin{thm}\label{veedist}
If $B$ is $\vee$-distributive in $\mathcal{H}(A)$ then $B$ is an ideal.
\end{thm}

\begin{proof}
Take $B\in\mathcal{H}(A)$ that is not an ideal, so we have $x\in A$ and $y\in B$ with $xy\notin B$.  As $y^*x^*xy\in B$, we must have $xyy^*x\notin B$ so, setting $a=xy$, we have $|a|^2=a^*a\in B$ but $|a^*|^2=aa^*\notin B$.  Define continuous functions $(g_n)$ on $\mathbb{R}_+$ which uniformly approach the identity and satisfy $g_n^{-1}\{0\}\supseteq[0,\delta_n)\neq\emptyset$, for all $n\in\mathbb{N}$.  Then $g_n(|a^*|)\rightarrow|a^*|\notin B$ which, as $B$ is closed, means $g_n(|a^*|)\notin B$, for some $n$.  Set $g=\lambda g_n$ for this $n$, where $\lambda\in(0,1/||a||)$, and define $f$, $d$ and $e$ as in the first paragraph of the proof of \autoref{pseudo}.  If we had $ee^*\in B$ then $g(|a^*|)^2=dd^*\in B$, by \eqref{dd*}, a contradiction.  Thus, by replacing $a$ with $e$ if necessary, we may further assume that $a^2=0$.

Now set $u=u_a$, $v=v_a$, $C=\overline{aAa^*}$ and $D=vA_av^*$.  As $aa^*\notin B$, we have $\phi\in A^*_+\setminus\{0\}$ with $\phi(aa^*)>0$ and $\phi[B]=\{0\}$.  Define $\psi\in A^*_+\setminus\{0\}$ by $\psi(z)=\phi(uzu^*)$.  Note $\phi$ extends uniquely to a normal state on $A''$ so $\psi$ is well defined (in fact, as $A$ is given the reduced atomic representation, there must be some $h\in H$ with $\phi=\phi_h$ and then we may set $\psi=\phi_{u^*h}$).  Note that $(1-u^*u)v=\frac{1}{\sqrt{2}}u=uv$ so $(1-u^*u)vzv(1-u^*u)=uvzv^*u^*$, for any $z\in A$, so $(1-u^*u)d(1-u^*u)=udu^*$, for any $d\in D$.  Thus if $d\in B\cap D$ then $udu^*\in B$ and hence $\psi(d)=\phi(udu^*)=0$, i.e. $\psi[B\cap D]=\{0\}$.  As $a^2=0$, $\psi[C]=\{0\}$ too and hence $\psi[(B\cap C)\vee(B\cap D)]=\{0\}$.  But $a^*a\in C\vee D$, by \eqref{Sasakieq2}, and $a^*a\in B$ too even though $\psi(a^*a)=\phi(ua^*au^*)=\phi(u|a||a|u^*)=\phi(aa^*)>0$ so \[(B\cap C)\vee(B\cap D)\neq B\cap(C\vee D).\]
\end{proof}

We can make $\mathcal{H}(A)$ a quantale (see \cite{Mulvey1986}) by defining $B\& C=B\cap\overline{\mathrm{span}}(ABA)$.  Moreover, this agrees with the usual quantale structure on the lattice $\mathcal{I}_R(A)$ of closed right ideals given by $I\& J=\overline{\mathrm{span}}(IJ)$, when we identify every $I\in\mathcal{I}_R(A)$ with $I\cap I^*\in\mathcal{H}(A)$ (see \cite{BorceuxBossche1986} Proposition 2).  By \autoref{veedist}, the ideals in $\mathcal{H}(A)$ can be identified purely from the order structure of $\mathcal{H}(A)$, so the lattice $\mathcal{H}(A)$ completely determines the quantale $\mathcal{H}(A)$.  As any postliminary $A$ is completely determined by the quantale $\mathcal{H}(A)$, by the theorem at the end of \cite{BorceuxRosickyBossche1989},\footnote{There seem to be some details missing in the proof of this theorem.  Specifically, we do not see how to get an isomorphism of q-spaces, i.e. an algebraic isomorphism of $\mathcal{B}(H)$ and $\mathcal{B}(H')$ identifying open projections, purely from a quantale isomorphism.  Hopefully this is just a lack of understanding on our part for, if not, it would put our stronger result \autoref{postlimlat} in some doubt.} we also get the following strengthening of \cite{BorceuxRosickyBossche1989} Proposition 5.

\begin{cor}\label{postlimlat}
Any postliminary $A$ is completely determined by the lattice $\mathcal{H}(A)$.
\end{cor}

It would be interesting to know if this can be extended to some broader class of C*-algebras.  Another natural structure to consider on $\mathcal{H}(A)$ would be the Peligrad-Zsid\'{o} equivalence relation $\sim_\mathrm{PZ}$ given in \cite{PeligradZsido2000}, and this might help in distinguishing more C*-algebras.  However, this would still not distinguish between $A$ and $A^\mathrm{op}$, even though these can be non-isomorphic C*-algebras (see \cite{Phillips2004}).  But perhaps $(\mathcal{H}(A),\sim_\mathrm{PZ})$ could still be a complete isomorphism invariant within a large class of (Elliot invariant) classifiable C*-algebras? 

Dual to \autoref{veedistdef}, we also have \emph{$\wedge$-distributivity} and \emph{$\bigwedge$-distributivity}.  Moreover, \autoref{veedist} can also be proved with $\wedge$-distributivity in place of $\vee$-distributivity.

\begin{thm}\label{wedgedist}
If $B$ is $\wedge$-distributive in $\mathcal{H}(A)$ then $B$ is an ideal.
\end{thm}

\begin{proof}
If $B\in\mathcal{H}(A)$ is not an ideal then, as in the proof of \autoref{veedist}, we have $a\in A$ with $a^*a\in B$ but $aa^*\notin B$, and we may set $u=u_a$, $v=v_a$, $C=\overline{aAa^*}$ and $D=vA_av^*$.  Then $C\subseteq\overline{a^*Aa}\vee D\subseteq B\vee D$ and hence $C\subseteq (B\vee C)\wedge(B\vee D)$, even though $C\cap D=\{0\}$ and hence $C\nsubseteq B=B\vee(C\wedge D)$.
\end{proof}

But, unlike with $\vee$-distributivity, even ideals in $\mathcal{H}(A)$ can fail to be $\wedge$-distributive.

\begin{xpl}\label{wedgedistxpl}
Let $A=C(\mathbb{N}\cup\{\infty\},M_2)$ and identify $A''$ with all bounded functions from $\mathbb{N}\cup\{\infty\}$ to $M_2$.  Define $p\in\mathcal{P}(A'\cap A'')^\circ$ by $p(n)=n\ \mathrm{mod}\ 2$ and $p(\infty)=0$.  Also define $q,r\in\mathcal{P}(A'\cap A'')^\circ$ by $q(n)=P_{1/n}$ and $r(n)=P_{(-1)^n/n}$ (and $q(\infty)=P_0=r(\infty)$) where
\[P_\theta=\begin{bmatrix}\sin\theta \\ \cos\theta\end{bmatrix}\begin{bmatrix}\sin\theta & \cos\theta\end{bmatrix}=\begin{bmatrix} \sin^2\theta & \sin\theta\cos\theta \\ \sin\theta\cos\theta & \cos^2\theta \end{bmatrix}\in\mathcal{P}(M_2).\]
Then $(q\wedge r)(n)=(n\ \mathrm{mod}\ 2)P_{1/n}$, for all $n\in\mathbb{N}$, and hence $(q\wedge r)^\circ(\infty)=0$ so $(p\vee(q\wedge r)^\circ)(\infty)=0$.  But $(p\vee q)(n)=1-(n\ \mathrm{mod}\ 2)P_{1/n+\pi/2}=(p\vee r)(n)$, for all $n\in\mathbb{N}$, and $(p\vee q)(\infty)=P_0=(p\vee r)(\infty)$, so $p\vee q=((p\vee q)\wedge(p\vee r))^\circ=p\vee r$ and hence \[((p\vee q)\wedge(p\vee r))^\circ(\infty)=P_0\neq0=(p\vee(q\wedge r)^\circ)(\infty).\]
\end{xpl}

This is somewhat surprising, given that when \emph{every} $p\in\mathbb{P}$ is $\vee$-distributive, \emph{every} $p\in\mathbb{P}$ is also $\vee$-distributive (see e.g. \cite{Blyth2005} \S5.1).  In this case we simply call $\mathbb{P}$ \emph{distributive}.  The $B\in\mathcal{H}(A)$ for which $\mathcal{H}(B)$ is distributive can be characterized in several different ways, as we now show (see also \cite{PeligradZsido2000} Lemma 2.6).

We denote the closed ideals of $A$ by $\mathcal{I}(A)$ and call $B\in\mathcal{H}(A)$ \emph{ideal-finite} if $\overline{\mathrm{span}}(ACA)=\overline{\mathrm{span}}(ABA)$ implies $C=B$, for all $C\in\mathcal{H}(B)$.

\begin{cor}\label{HAcomm}
For any $B\in\mathcal{H}(A)$, the following are equivalent.
\begin{enumerate}
\item\label{HAcomm1} $B$ is commutative.
\item\label{HAcomm2} $\mathcal{H}(B)=\mathcal{I}(B)$.
\item\label{HAcomm3} $B$ is ideal-finite.
\item\label{HAcomm4} $\mathcal{H}(B)$ is distributive.
\end{enumerate}
\end{cor}

\begin{proof}\
\begin{itemize}
\item[\eqref{HAcomm1}$\Rightarrow$\eqref{HAcomm2}] If $B$ is commutative and $C\in\mathcal{H}(B)$ then $BC=CB=BC\cap CB\subseteq C$.
\item[\eqref{HAcomm2}$\Rightarrow$\eqref{HAcomm3}] If $C\in\mathcal{I}(B)$ then $C=B\cap\overline{\mathrm{span}}(ACA)$, so if $\overline{\mathrm{span}}(ACA)=\overline{\mathrm{span}}(ABA)$ then $C=B\cap\overline{\mathrm{span}}(ABA)=B$.
\item[\eqref{HAcomm2}$\Rightarrow$\eqref{HAcomm4}] $C,D,E\in\mathcal{I}(B)\Rightarrow C\wedge(D\vee E)=C\overline{D+E}=\overline{CD+CE}=(C\wedge D)\vee(C\wedge E)$.
\item[\eqref{HAcomm4}$\Rightarrow$\eqref{HAcomm2}] See \autoref{veedist}.
\item[\eqref{HAcomm3}$\Rightarrow$\eqref{HAcomm1}] If $B$ is not commutative then $\mathrm{rank}(\pi(p_B))>1$, for some $\pi\in\hat{A}$.  Taking $v\in\mathcal{R}(\pi(p_B))\setminus\{0\}$, we have $C=\{b\in B:\pi(b)v=0=\pi(b^*)v\}\subsetneqq B$, even though $\overline{\mathrm{span}}(ACA)=\overline{\mathrm{span}}(ABA)$.
\end{itemize}
\end{proof}

\begin{qst}
Are the $\wedge$-distributive elements of $\mathcal{H}(A)$ precisely those of the form $I\oplus pAp$ where $I\in\mathcal{I}(A)$ is commutative and $p\in\mathcal{P}(\mathcal{M}(A))\cap A'$?
\end{qst}

While we do not know the answer to this question, or even if there is any algebraic characterization of $\wedge$-distributivity in $\mathcal{H}(A)$, we can obtain a number of characterizations of $\bigwedge$-distributivity.  First, let us introduce the order theoretic notion of centrality.

Note that, for any $p$ and $q$ in a poset $\mathbb{P}$, $[p,q]$ denotes the interval they define, i.e. $[p,q]=\{r\in\mathbb{P}:p\leq r\leq q\}$.

\begin{dfn}
We call $p\in\mathbb{P}$ \emph{central} if $\mathbb{P}\cong[0,p]\times[0,p^\perp]$, for some $p^\perp\in\mathbb{P}$, via
\[(q,r)\mapsto q\vee r\quad\textrm{and}\quad s\mapsto(s\wedge p,s\wedge p^\perp).\]
\end{dfn}

As with complements, the central elements of $\mathscr{P}(X)^\circ$ are precisely the clopen sets.  However, just like with $\wedge$-pseudocomplements, central elements in $\mathcal{H}(A)$ must be ideals.  In fact, for $p\in\mathcal{P}(A'')^\circ$, the algebraic notion of central, in $\mathcal{M}(A)$, coincides with the order theoretic notion just defined, as well as to the dual of several other order theoretic notions examined so far.

\begin{thm}\label{ce}
For any $p$ in $\mathcal{P}(A'')^\circ$, the following are equivalent.
\begin{enumerate}
\item\label{p+pperp} $A=A_p\oplus A_{p^\perp}$.
\item\label{MAcapA'} $p\in\mathcal{M}(A)\cap A'$.
\item\label{pcen} $p$ is central.
\item\label{bigmdist} $p$ is $\bigwedge$-distributive.
\item\label{jdcomp} $p$ is/has a $\wedge$-distributive complement.
\item\label{jpseudo} $p$ is/has a $\vee$-pseudocomplement.
\end{enumerate}
\end{thm}

\begin{proof}
We immediately see that \eqref{pcen}$\ \Rightarrow\ p$ is a $\vee$-pseudocomplement and
\[\eqref{p+pperp}\ \Rightarrow\ \eqref{MAcapA'}\ \Rightarrow\ \eqref{pcen}\ \Rightarrow\ \eqref{bigmdist}\ \Rightarrow\ p\textrm{ has a $\vee$-pseudocomplement}.\]
Also \autoref{trieq} \eqref{vsemi} $\Rightarrow$ \eqref{dsum} yields \eqref{jpseudo} $\Rightarrow$ \eqref{p+pperp} here.  Another immediate implication is \eqref{pcen} $\Rightarrow$ \eqref{jdcomp}, while \autoref{wedgedist} and \autoref{trieq} \eqref{ideal} $\Rightarrow$ \eqref{dsum} yields \eqref{jdcomp} $\Rightarrow$ \eqref{p+pperp} here.
\end{proof}

Before moving on, let us point out that the centre of $\mathcal{H}(A)$ may not be complete, even though $\mathcal{H}(A)$ itself is a complete lattice.  Indeed, the $p$ in \autoref{wedgedistxpl} is not in $A$, even though it is a (countable) supremum and infimum (taken in $\mathscr{P}(A'')^\circ$) of projections in $A\cap A'$.  Whether this could happen in any complete lattice was mentioned as a question of S. Holland in \cite{Birkhoff1967} Ch 5 Problem 34, with examples and related theory given in \cite{Jakubik1973} and \cite{Janowitz1978}.  \autoref{wedgedistxpl} shows that this situation crops up quite naturally in $\mathcal{H}(A)$, even for quite elementary C*-algebras $A$.

\section{Separativity and *-Annihilators}\label{S&*A}

Dual to subfitness, we have separativity, i.e. $p\in\mathbb{P}$ is \emph{separative} if $p$ is $\wedge$-separated from every $q\nleq p$.  We also call $\mathbb{P}$ separative (see \cite{Kunen1980}) when every $p\in\mathbb{P}$ is separative, although some authors would call such $\mathbb{P}$ SSC.  Instead, we work with the following slightly weaker definition of SSC (which agrees with the original definition in \cite{MaedaMaeda1970})

\begin{dfn}\label{SSCSep}
We call $p$ \emph{section $\wedge$-semicomplemented} or \emph{SSC} if $p$ is $\wedge$-separated from every $q>p$.  We call $\mathbb{P}$ itself SSC when every $p\in\mathbb{P}$ is SSC.
\end{dfn}

Yet again, let us consider the lattice of open sets $\mathscr{P}(X)^\circ$ of a topological space $X$.  If $O\in\mathscr{P}(X)^\circ$ is regular then, for any $N\in\mathscr{P}(X)^\circ$ with $N\nsubseteq O$, we must also have $N\nsubseteq\overline{O}$ (because $N\subseteq\overline{O}$ would imply $N=N^\circ\subseteq\overline{O}^\circ=O$) and hence $N\setminus\overline{O}$ is a non-empty $\wedge$-semicomplement of $O$ in $\mathscr{P}(X)^\circ$.  While if $O$ is not regular then $O\subsetneqq\overline{O}^\circ$, even though the definition of closure means there is no non-empty open $N\subseteq\overline{O}$ with $O\cap N=\emptyset$.  So, as with $\wedge$-pseudocomplements, the SSC/separative elements of $\mathscr{P}(X)^\circ$ are precisely the regular open subsets.  However, unlike $\wedge$-pseudocomplements, SSC/separative elements of $\mathcal{H}(A)$ need not be ideals, which naturally leads to the following question \textendash\, are the SSC/separative elements of $\mathcal{H}(A)$ precisely the *-annihilators?

This time, the answer is no in general (see \autoref{C01K} below).  However, we can prove one direction, namely that *-annihilators in $\mathcal{H}(A)$ are necessarily separative, and in fact satisfy a strong version of being SSC too.  For this, we first require a number of slightly technical spectral projection inequalities.\footnote{Many of the results that follow first appeared in the preprint \cite{Bice2013}.}

\begin{lem}\label{c1-e}
For $\epsilon,\lambda>0$, there exists $\delta>0$ such that, whenever $b,c\in\mathcal{B}(H)^1_+$, $c\leq q\in\mathcal{P}(\mathcal{B}(H))$ and $||bq||^2\leq\lambda+\delta$, we have
\begin{equation}\label{c1-eeq}
||c_{[0,1-\epsilon]}(cb^2c)_{[\lambda-\delta,1]}||\leq\epsilon.
\end{equation}
\end{lem}

\begin{proof}
If $\epsilon\geq1$ then \eqref{c1-eeq} holds trivially, so assume $\epsilon<1$.  For any $v\in H$,
\begin{eqnarray}
||cv||^2 &=& ||cc_{[0,1-\epsilon]}v||^2+||cc_{(1-\epsilon,1]}v||^2\nonumber\\
&\leq& (1-\epsilon)^2||c_{[0,1-\epsilon]}v||^2+||c_{(1-\epsilon,1]}v||^2\nonumber\\
&\leq& ||v||^2-\epsilon(2-\epsilon)||c_{[0,1-\epsilon]}v||^2\nonumber\\
&\leq& ||v||^2-\epsilon||c_{[0,1-\epsilon]}v||^2\qquad\textrm{(as $\epsilon\leq1$)}.\label{cv}
\end{eqnarray}
As $c\leq q$, we have $q^\perp cq^\perp\leq q^\perp qq^\perp=0$ and hence $q^\perp c=0$, i.e. $c=qc$.  Thus, for $v\in\mathcal{R}((cb^2c)_{[\lambda-\delta,1]})$,
\begin{eqnarray*}
(\lambda-\delta)||v||^2 &\leq& \langle cb^2c v,v\rangle\\
&=& ||bcv||^2\\
&\leq& ||bq||^2||cv||^2,\textrm{ as }c=qc,\\
&\leq& (\lambda+\delta)(||v||^2-\epsilon||c_{[0,1-\epsilon]}v||^2),\textrm{ by \eqref{cv}, so}\\
(\lambda+\delta)\epsilon||c_{[0,1-\epsilon]}v||^2 &\leq& 2\delta||v||^2\textrm{ and}\\
||c_{[0,1-\epsilon]}v||^2 &\leq& 2\delta||v||^2/(\lambda\epsilon),
\end{eqnarray*}
which immediately yields \eqref{c1-eeq}, for $\delta\leq\lambda\epsilon^3/2$.
\end{proof}

\begin{lem}\label{1-c}
For $\epsilon,\lambda>0$, there exists $\delta>0$ such that, whenever $b,c\in\mathcal{B}(H)^1_+$, $c\leq q\in\mathcal{P}(\mathcal{B}(H))$ and $||bq||^2\leq\lambda+\delta$, we have
\begin{equation}\label{1-ceq}
||(1-c)(cb^2c)_{[\lambda-\delta,1]}||\leq\epsilon.
\end{equation}
\end{lem}

\begin{proof}
Replacing $\epsilon$ with $\epsilon/\sqrt{2}$ in \autoref{c1-e}, we obtain $\delta>0$ such that, for any $v\in\mathcal{R}((cb^2c)_{[\lambda-\delta,1]})$,
\begin{eqnarray*}
||(1-c)v||^2 &=& ||(1-c)c_{[0,1-\epsilon/\sqrt{2})}v||^2+||(1-c)c_{[1-\epsilon/\sqrt{2},1]}v||^2\\
&\leq& ||c_{[0,1-\epsilon/\sqrt{2})}v||^2+\epsilon^2||c_{[1-\epsilon/\sqrt{2},1]}v||^2/2\\
&\leq& \epsilon^2||v||^2/2+\epsilon^2||v||^2/2,\textrm{ by \eqref{c1-eeq}}.
\end{eqnarray*}
\end{proof}

The following result generalizes \cite{Bice2009} Lemma 5.3.

\begin{lem}\label{lem2}
For $\epsilon,\lambda>0$, there exists $\delta>0$ such that, whenever $b,c\in\mathcal{B}(H)^1_+$, $c\leq q\in\mathcal{P}(\mathcal{B}(H))$ and $||bq||^2\leq\lambda+\delta$, we have
\begin{equation}\label{lem2eq}
||b_{[0,\sqrt{\delta}]}(cb^2c)_{[\lambda-\delta,1]}||^2\leq1-\lambda+\epsilon.
\end{equation}
\end{lem}

\begin{proof}  Let $\delta>0$ be that obtained in \autoref{1-c} from replacing $\epsilon$ with $\epsilon/4$.  If necessary, replace $\delta$ with a smaller non-zero number so that we also have
\begin{equation}\label{deltaeq}
(1-\lambda+\delta+\epsilon/2)/(1-\delta)\leq1-\lambda+\epsilon.
\end{equation}
Then, for all $v\in\mathcal{R}((cb^2c)_{[\lambda-\delta,1]})$,
\begin{eqnarray*}
(\lambda-\delta)||v||^2 &\leq& \langle cb^2cv,v\rangle\\
&=& \langle b^2cv,cv\rangle\\
&\leq& \langle b^2v,v\rangle+\epsilon||v||^2/2,\textrm{ by \eqref{1-ceq}},\\
&=& \langle b^2b_{[0,\sqrt{\delta}]}v,v\rangle+\langle b^2b_{(\sqrt{\delta},1]}v,v\rangle+\epsilon||v||^2/2\\
&\leq& \delta\langle b_{[0,\sqrt{\delta}]}v,v\rangle+\langle b_{(\sqrt{\delta},1]}v,v\rangle+\epsilon||v||^2/2\\
&=& \delta||b_{[0,\sqrt{\delta}]}v||^2+(||v||^2-||b_{[0,\sqrt{\delta}]}v||^2)+\epsilon||v||^2/2,\textrm{ so}\\
\quad(1-\delta)||b_{[0,\sqrt{\delta}]}v||^2 &\leq& (1-\lambda+\delta+\epsilon/2)||v||^2,\textrm{ and hence}\\
||b_{[0,\sqrt{\delta}]}v||^2 &\leq& (1-\lambda+\epsilon)||v||^2,\textrm{ by \eqref{deltaeq}}.
\end{eqnarray*}
\end{proof}

\begin{dfn}
For $\epsilon\in(0,1]$, we call $B\in\mathcal{H}(A)$ \emph{$\epsilon$-SSC} if, whenever $B\subsetneqq C\in\mathcal{H}(A)$, there exists $D\in\mathcal{H}(C)$ with $||p_Bp_D||<\epsilon$.
\end{dfn}

So the smaller $\epsilon$ is, the stronger the $\epsilon$-SSC property is, and if $B\in\mathcal{H}(A)$ is even $1$-SSC then it is SSC, according to \autoref{SSCSep}.  So the following result answers one direction of our original question.

\begin{thm}\label{1sep}
Any $B\in\mathcal{H}(A)$ with $B=B^{\perp\perp}$ is $1$-SSC.
\end{thm}

\begin{proof}
Take $C\in\mathcal{H}(A)$ with $B\subsetneqq C$, so we have $c\in C^1_+\setminus B$.  This means we have $b\in B^{\perp1}_+$ with $bc\neq0$, and hence $bq\neq0$, where $q=c_{(0,1]}$.  Set $\lambda=||bq||^2$, take positive $\epsilon<\lambda$ and let $\delta>0$ be that obtained in \autoref{lem2}.  Note that we may now assume that $||bc||^2>\lambda-\delta$ by replacing $c$ with $f(c)$, where $f$ is a continuous function on $[0,1]$ with $[\mu,1]\subseteq f^{-1}\{1\}$, for sufficiently small $\mu>0$.  Set $p=(cb^2c)_{(\lambda-\delta,1]}\in\mathcal{P}(C'')^\circ\setminus\{0\}$.  As $p_B\leq b_{\{0\}}\leq b_{[0,\sqrt{\delta}]}$, \autoref{lem2} yields \[||p_Bp||^2\leq||b_{[0,\sqrt{\delta}]}(cb^2c)_{[\lambda-\delta,1]}||^2\leq1-\lambda+\epsilon<1.\]
\end{proof}

In fact, we can do much better than $1$-SSC, but first we need some more results.

\begin{lem}\label{lem3}
For $\epsilon,\lambda>0$, there exists $\delta>0$ such that, whenever $b,c\in\mathcal{B}(H)^1_+$, $p,q\in\mathcal{P}(\mathcal{B}(H))$, $b\leq p$, $c\leq q$ and $||pq||^2\leq\lambda+\delta$, we have
\begin{equation}\label{lem3eq}
||p(cb^2c)_{[\lambda-\delta,1]}||^2\leq\lambda+\epsilon.
\end{equation}
\end{lem}

\begin{proof}
Let $\delta>0$ be that obtained in \autoref{1-c} with $\epsilon$ replaced with $\epsilon/4$.  If necessary, decrease $\delta$ so that $\delta\leq\epsilon/2$.  Then, for $v\in\mathcal{R}((cb^2c)_{[\lambda-\delta,1]})$,
\begin{eqnarray*}
||pv||^2 &=& \langle pv,pv\rangle\\
&\leq& \langle pcv,pcv\rangle+\epsilon||v||^2/2\\
&\leq& (||pc||^2+\epsilon/2)||v||^2\\
&\leq& (\lambda+\delta+\epsilon/2)||v||^2.\\
&\leq& (\lambda+\epsilon)||v||^2.
\end{eqnarray*}
\end{proof}

For use in the next result, note that whenever $p,q\in\mathcal{P}(A)$ and $p\neq0$,
\begin{equation}\label{Pythag}
||pq||^2+||pq^\perp||^2\geq1.
\end{equation}
For simply take $v\in\mathcal{R}(p)\backslash\{0\}$ and note that \[||v||^2=||qv||^2+||q^\perp v||^2=||qpv||^2+||q^\perp pv||^2\leq(||qp||^2+||q^\perp p||^2)||v||^2.\]
Also note that 
\begin{equation}\label{pnearq}
||pq^\perp||^2\leq\lambda\quad\Leftrightarrow\quad pq^\perp p\leq\lambda p\quad\Leftrightarrow\quad(1-\lambda)p\leq pqp.
\end{equation}
In fact, the following result is a natural modulo-$\epsilon$ generalization of \eqref{Pythag} from $\mathcal{P}(A)$ to $\mathcal{H}(A)$, where $B$, $C$ and $D$ correspond to $q$, $p$ and $q^\perp$ respectively.

\begin{thm}\label{septhm}
For $\epsilon>0$ and $B,C\in\mathcal{H}(A)\setminus\{0\}$ with $||p_Bp_C||^2=\lambda<1$, there exists $D\in\mathcal{H}(A)$ with $||p_Bp_D||\leq\epsilon$ and $||p_Cp_D||^2\geq1-\lambda-\epsilon$.
\end{thm}

\begin{proof} Choose $\delta>0$ small enough that it satisfies \autoref{lem2} and \autoref{lem3} with $\epsilon$ replaced by some $\mu>0$, to be determined later.  Take $c\in C^1_+$ and $b\in B^1_+$ with $||bc||^2>\lambda-\delta/2$.  Take $c'\in C^1_+$ with $(cb^2c)_{[\lambda-\delta/2,1]}\leq c'\leq(cb^2c)_{[\lambda-\delta,1]}$ and let $a=(1-f(b))c'^2(1-f(b))$, where $f$ is continuous, $0$ on $[0,\delta/2]$ and $1$ on $[\delta,1]$, so
\begin{eqnarray}
||a|| &=& ||(1-f(b))c'||^2\nonumber\\
&\geq& ||b_{\{0\}}(cb^2c)_{[\lambda-\delta/2,1]}||^2\nonumber\\
&\geq& 1-||b_{(0,1]}(cb^2c)_{[\lambda-\delta,1]}||^2,\textrm{ by \eqref{Pythag}}\nonumber\\
&\geq& 1-\lambda-\mu,\textrm{ by \eqref{lem3eq}.}\label{||s||}
\end{eqnarray}
In particular, $||a||>0$ as long as $\mu<1-\lambda$, and we may define $a'=||a||^{-1}a$.

By \eqref{lem2eq}, we have
\begin{equation}\label{anothereq}
||b_{[0,\sqrt{\delta}]}c'_{(0,1]}||^2\leq||b_{[0,\sqrt{\delta}]}(cb^2c)_{[\lambda-\delta,1]}||^2\leq1-\lambda+\mu
\end{equation}
and so, by \eqref{pnearq},
\begin{equation}\label{BDeq}
c'_{(0,1]}b_{[\delta,1]}c'_{(0,1]}\geq c'_{(0,1]}b_{[\sqrt{\delta},1]}c'_{(0,1]}\geq(\lambda-\mu)c'_{(0,1]}.
\end{equation}
Thus
\begin{eqnarray*}
(1-\lambda-\mu)||p_Ba'p_B|| &\leq& ||a||||p_Ba'p_B||,\textrm{ by \eqref{||s||}}\\
&=& ||p_Bap_B||\\
&=& ||p_B(1-f(b))c'^2(1-f(b))p_B||\\
&=& ||(p_B-f(b))c'||^2\\
&=& ||c'(p_B-f(b))^2c'||\\
&\leq& ||c'(p_B-f(b))c'||\\
&=& ||c'(p_Cp_Bp_C-c'_{(0,1]}b_{[\delta,1]}c'_{(0,1]})c'||\\
&\leq& ||c'(\lambda-\lambda+\mu)c'||,\textrm{ by \eqref{BDeq} and }||p_Bp_C||^2=\lambda\\
&\leq& \mu,\textrm{ and hence},\\
||p_Ba'p_B|| &\leq& \mu/(1-\lambda-\mu).
\end{eqnarray*}
Set $p=a'_{(1-\mu,1]}\in\mathcal{P}(A'')^\circ$ and $D=A_p$.  Now
\[||p_Bp||^2=||p_Bpp_B||\leq||p_Ba'p_B||/(1-\mu)\leq\mu/((1-\lambda-\mu)(1-\mu))\]
so, as long as $\mu>0$ was chosen sufficiently small, $||p_Bp||\leq\epsilon$.

Also, $||(1-f(b))_{(0,1]}c'||^2\leq||b_{[0,\sqrt{\delta}]}c'_{(0,1]}||^2\leq1-\lambda+\mu$, by \eqref{anothereq}.  This means, as long as we chose $\mu$ at least half as small as the $\delta$ obtained in \autoref{lem2} (from the given $\epsilon$) we can apply \eqref{lem2eq} with $b$, $c$ and $\lambda$ replaced by $c'$, $1-f(b)$ and $1-\lambda$ to get
\[||c'_{[0,\sqrt{\mu}]}a'_{(1-\mu,1]}||^2\leq||c'_{[0,\sqrt{2\mu}]}((1-f_\delta(b))c'^2(1-f_\delta(b)))_{(1-\lambda-2\mu,1]}||^2\leq\lambda+\epsilon,\]
(the first inequality follows from \eqref{||s||} and the fact $(1-\mu)(1-\lambda-\mu)\geq1-\lambda-2\mu$).  Thus $||p_Cp||^2\geq||c'_{[\sqrt{\mu},1]}a'_{(1-\mu,1]}||^2\geq1-\lambda-\epsilon$.
\end{proof}

For any $1$-SSC $B\in\mathcal{H}(A)$, the first part of the above result can be applied within any hereditary C*-subalgebra containing $B$ to show that $B$ must actually be $\epsilon$-SSC, for any $\epsilon>0$.  So \autoref{1sep} can immediately be strengthened as follows.

\begin{cor}\label{epsep}
Any $B\in\mathcal{H}(A)$ with $B=B^{\perp\perp}$ is $\epsilon$-SSC, for all $\epsilon>0$.
\end{cor}

It is natural to wonder if this can be strengthened just a little more to bring $\epsilon$ down to $0$, i.e. to show that whenever $B,C\in\mathcal{H}(A)$, $B=B^{\perp\perp}$ and $B\subsetneqq C$, we have $C\cap B^\perp\neq\{0\}$ (such a $B$ might well be called \emph{section $\perp$-semicomplemented} or \emph{orthomodular}).  The following example shows that this is not possible in general.

\begin{xpl}\label{C01K}
Let $A=C([0,1],\mathcal{K}(H))$, where $\mathcal{K}(H)$ is the C*-algebra of compact operators on a separable infinite dimensional Hilbert space $H$.  Identify $A''$ (in the atomic representation) with all bounded functions from $[0,1]$ to $\mathcal{B}(H)$.  Now let $p_n\in A$ be the rank $1$ projection onto $\mathbb{C}e_n$, for each $n\in\mathbb{N}$, where $(e_n)$ is an orthonormal basis for $H$.  Also let $(r_n)$ enumerate a countable dense subset of $(0,1)$ and let $\chi_S$ denote the characteristic function of $S\subseteq[0,1]$.  Consider $p=\bigvee\chi_{[0,r_n)}p_n\in\mathcal{P}(A'')^\circ$ and $p'=\bigvee\chi_{(r_n,1]}p_n\in\mathcal{P}(A'')^\circ$.  For any $a\in A_p^{\perp1}$, $p_na(x)=0$, for all $x\in[0,r_n)$ which, as $a$ is continuous, means $p_na(r_n)=0$ too so $a\leq p'$.  Thus $A_{p'}=A_p^\perp$ and, likewise, $A_p=A_{p'}^\perp=A_p^{\perp\perp}$.

Now let $q$ be the (constant) rank one projection onto $v=\sum2^{-n}e_n$ and consider $p\vee q\in\mathcal{P}(A'')^\circ$.  As $q\in A_{p\vee q}\setminus A_p$, we certainly have $A_p\subsetneqq A_{p\vee q}$.  But $p\vee q-p$ is a rank 1 projection on $(0,1)$ which is discontinuous on the dense subset $(r_n)$, so $A_{p\vee q}\cap A_p^\perp=A_{p\vee q-p}=\{0\}$.
\end{xpl}

Furthermore, the $A_p$ above is SSC in $\mathcal{H}(A)$, by \autoref{1sep}, and so certainly SSC in $\mathcal{H}(A_{p\vee q})$.  But we just showed that $A_p$ is not a *-annihilator in $A_{p\vee q}$, and thus we can not hope to use the SSC property to characterize *-annihilators in general.

\begin{qst}
Is there an order theoretic characterization of *-annihilators in $\mathcal{H}(A)$?
\end{qst}

Still considering \autoref{C01K} above, note that, as $A_p$ is $1$-SSC in $A_{p\vee q}$, we have $b\in B^1_+\in\mathcal{H}(A_{p\vee q})$ with $||bp||\leq||p_Bp||<1=||b||$, despite the fact $p$ is dense in $p\vee q$ (equivalently, $(p\vee q-p)^\circ=0$), i.e. $p$ is \emph{non-regular} in the sense of \cite{Tomita1959} (this concept of regularity has little to do with the topological regularity of open sets discussed earlier).  The question of whether there exist open dense non-regular projections was mentioned as an open problem in \cite{PeligradZsido2000}, and the first examples were given in \cite{AkemannEilers2002} (which inspired our construction of \autoref{C01K}).  In \cite{AkemannEilers2002}, a constant $\gamma$ was even defined to measure the degree of regularity of an open dense projection $p$, essentially by \[\gamma(p)=\inf_{q\in\mathcal{P}(A'')^\circ}||pq||,\] where $\gamma(p)=1$ means $p$ is regular and lower $\gamma(p)$ values signify lower regularity.  However, \autoref{septhm} shows that if $\gamma(p)<1$ then, in fact, $\gamma(p)=0$, i.e. any non-regular open dense projection must actually be as non-regular as possible.

\section{The *-Annihilator Ortholattice}\label{*AO}

We have just seen in the previous section (and \S\ref{*AI}) that *-annihilators have special properties within $\mathcal{H}(A)$, and one might guess they could be worthy of study in their own right.  Indeed, a surprisingly detailed *-annihilator theory, closely resembling the basic theory of projections in von Neumann algebras, can be developed even in the much broader context of *-semigroups (see \cite{Bice2014c}).  Here we investigate what more can be said about them in the C*-algebra context, and how closely related the *-annihilator ortholattice $\mathscr{P}(A)^\perp$ is to the hereditary C*-subalgebra lattice $\mathcal{H}(A)$.

Again consider a topological space $X$, but this time assume it also satisfies the $T_3$ separation axiom, i.e. any disjoint point and closed subset have disjoint neighbourhoods.  Thus, whenever $x\in O\in\mathscr{P}(X)^\circ$, we have $N\in\mathscr{P}(X)^\circ$ with $x\in N$ and $\overline{N}\subseteq O$, and hence $x\in\overline{N}^\circ\subseteq O$.  As $x\in O$ was arbitrary, \[O=\bigcup\{N\subseteq X:N=\overline{N}^\circ\subseteq O\},\] i.e. the regular open subsets of $X$ are $\bigvee$-dense in $\mathscr{P}(X)^\circ$.  As any locally compact Hausdorff space is $T_3$, it follows that $\mathscr{P}(A)^\perp$ is $\bigvee$-dense in $\mathcal{H}(A)$ whenever $A$ is commutative.  Yet again, the commutativity assumption here is unnecessary, as we now show.

First note that, for any $a\in A^1_+$ and $\lambda\in(0,1)$, $a_{[0,\lambda)}$ is open.  Indeed, letting $f$ be any continuous function on $[0,1]$ that is non-zero on $[0,\lambda)$ and $0$ on $[\lambda,1]$, we have $a_{[0,\lambda)}=f(a)_{(0,1]}=p_{\overline{f(a)Af(a)}}$.  The only slight problem is that $f(a)\notin A$ when $A$ is not unital, but we still have $a_{[0,\lambda)}=p_B$ for some $B\in\mathcal{H}(A)$, by \cite{Pedersen1979} Proposition 3.11.9.  As $A_{a_{[0,\lambda)}}\subseteq A_{a_{(\lambda,1]}}^\perp$, we have $A_{a_{(\lambda,1]}}^{\perp\perp}\subseteq A_{a_{[0,\lambda)}}^\perp\subseteq A_{a_{[\lambda,1]}}$.  In other words, letting $p=p_{A_{a_{(\lambda,1]}}^{\perp\perp}}\in\mathcal{P}(A'')^\circ$, we have $A_p^{\perp\perp}=A_p$ and
\begin{equation}\label{annproj}
a_{(\lambda,1]}\leq p\leq a_{[\lambda,1]}.
\end{equation}

Now take any $B\in\mathcal{H}(A)$ and let $(f_n)$ be a sequence of continuous functions on $[0,1]$ uniformly approaching the identity with $[0,1/n]\subseteq f_n^{-1}\{0\}$, for all $n\in\mathbb{N}$.  Then, for any $b\in B^1_+$ and $n\in\mathbb{N}$, we have $p\in\mathcal{P}(A'')^\circ$ with $A_p=A_p^{\perp\perp}$ and $b_{(1/n,1]}\leq p\leq b_{[1/n,1]}$ and hence $f_n(b)\in A_p\subseteq\overline{bAb}\subseteq B$.  As $f_n(b)\rightarrow b$, we have \[b\in B_\vee=\bigvee(\mathscr{P}(A)^\perp\cap\mathcal{H}(B))=\bigvee\{C\in\mathcal{H}(B):C^{\perp\perp}=C\}.\] As $b$ was arbitrary, $B=B_\vee$.  As $B$ was arbitrary, $\mathscr{P}(A)^\perp$ is $\bigvee$-dense in $\mathcal{H}(A)$.

Any complete lattice $\mathbb{Q}$ that is $\bigvee$-dense in another poset $\mathbb{P}$ must in fact be a complete $\wedge$-sublattice of $\mathbb{P}$, i.e. infimums in $\mathbb{Q}$ are also valid in $\mathbb{P}$.  For if $q$ is the infimum of $S$ in $\mathbb{Q}$ and $p\in\mathbb{P}$ satisfies $p\leq s$, for all $s\in S$, then
\[p=\bigvee\{r\in\mathbb{Q}:r\leq p\}\leq\bigvee\{r\in\mathbb{Q}:\forall s\in S(r\leq s)\}=q.\]
In particular, $\mathscr{P}(A)^\perp$ is a complete $\wedge$-sublattice of $\mathcal{H}(A)$, although this can also be seen directly from $\bigcap_\alpha(B_\alpha^\perp)=(\bigcup_\alpha B_\alpha)^\perp$.  However, it is important to note that $\mathscr{P}(A)^\perp$ is not a $\vee$-sublattice of $\mathcal{H}(A)$, even for commutative $A$.  For example, $[0,\frac{1}{2})$ and $(\frac{1}{2},1]$ are regular open subsets of $[0,1]$ even though $[0,\frac{1}{2})\cup(\frac{1}{2},1]$ is not.

\begin{prp}\label{SSC=sep}
Assume $\mathbb{Q}$ is a $\bigvee$-dense $\wedge$-sublattice of SSC elements in $\mathbb{P}$.  Then every $q\in\mathbb{Q}$ is separative in both $\mathbb{P}$ and $\mathbb{Q}$.
\end{prp}

\begin{proof}
Take $q\in\mathbb{Q}\setminus\{0\}$ and $p\in\mathbb{P}\setminus\{0\}$ with $p\nleq q$.  As $\mathbb{Q}$ is $\bigvee$-dense in $\mathbb{P}$, we have $r\in\mathbb{Q}\setminus\{0\}$ with $r\leq p$ and $r\nleq q$, and hence $q\wedge r<r$.  But $\mathbb{Q}$ is a $\wedge$-sublattice so $q\wedge r\in\mathbb{Q}$.  As elements of $\mathbb{Q}$ are SSC in $\mathbb{P}$, we have $s\in\mathbb{P}\setminus\{0\}$ with $s\leq r\leq p$ and $s\wedge q=s\wedge q\wedge r=0$.  Thus $q$ is separative in $\mathbb{P}$ and, again using join density (actually order density would be enough), we can replace $s$ with an element of $\mathbb{Q}$ to show that $q$ is separative in $\mathbb{Q}$ too.
\end{proof}

In particular, any SSC $\wedge$-lattice is separative.  Also, \autoref{1sep} now immediately yields the following.

\begin{cor}
Every $B\in\mathscr{P}(A)^\perp$ is separative in both $\mathcal{H}(A)$ and $\mathscr{P}(A)^\perp$.
\end{cor}

Next we examine the elements of $\mathscr{P}(A)^\perp$ with special order properties, as in the previous sections.  Note that, as $\mathscr{P}(A)^\perp$ is not a $\vee$-sublattice of $\mathcal{H}(A)$, various lattice theoretic concepts can potentially have very different meanings in $\mathscr{P}(A)^\perp$ and $\mathcal{H}(A)$.  For example, complements in $\mathcal{H}(A)$ are quite special, while every $B\in\mathscr{P}(A)^\perp$ has a complement in $\mathscr{P}(A)^\perp$, namely $B^\perp$.  In fact, $^\perp$ is an \emph{orthocomplementation} on $\mathcal{P}(A)^\perp$ (while it is merely a \emph{Galois $\wedge$-semicomplementation} on $\mathcal{H}(A)$), which makes $\mathcal{P}(A)^\perp$ an \emph{ortholattice}, giving us access to ortholattice theory and concepts, like the following.

\begin{dfn}\footnote{For more information on the commutativity relation $\mathrm{C}$, particularly in orthomodular lattices, see \cite{Kalmbach1983} or \cite{Beran1985}.  On the other hand, the Elkan relation $\mathrm{E}$ does not seem to have been formally defined before, although a similar global condition, \emph{Elkan's Law}, was studied in \cite{DudekKondo2008}.}
In an ortholattice $\mathbb{P}$, we define the \emph{commutativity} relation $\mathrm{C}$ by
\[p\mathrm{C}q\quad\Leftrightarrow\quad p=(p\wedge q)\vee(p\wedge q^\perp).\]
We also define the $\emph{Elkan}$ relation $\mathrm{E}$ by
\[p\mathrm{E}q\quad\Leftrightarrow\quad p\vee q=(p\wedge q^\perp)\vee q.\]
\end{dfn}

\begin{prp}\label{so}
For $q$ in a separative ortholattice $\mathbb{P}$, the following are equivalent.
\begin{enumerate}
\item\label{so1} $q$ is central.
\item\label{so2} $q$ is $\vee$-distributive.
\item\label{so3} $p\mathrm{C}q$, for all $p\in\mathbb{P}$.
\item\label{so4} $p\mathrm{E}q$, for all $p\in\mathbb{P}$.
\item\label{so5} $q$ is a $\wedge$-pseudocomplement of $q^\perp$.
\end{enumerate}
\end{prp}

\begin{proof}
Even without separativity, we immediately see that
\[\eqref{so1}\quad\Rightarrow\quad\eqref{so2}\textrm{ or }\eqref{so3}\quad\Rightarrow\quad\eqref{so4}\quad\Rightarrow\quad\eqref{so5},\]
and $\eqref{so3}\Rightarrow\eqref{so1}$, by \cite{MacLaren1964} Theorem 3.2.  While if \eqref{so3} fails then, for some $p\in\mathbb{P}$, we have $(p\wedge q)\vee(p\wedge q^\perp)<p$.  If $\mathbb{P}$ is separative/SSC, then we have non-zero $r\leq p$ with $r\wedge q=r\wedge p\wedge q\leq r\wedge((p\wedge q)\vee(p\wedge q^\perp))=0$.  Likewise, $r\wedge q^\perp=0$, which means $q$ is not a $\wedge$-pseudocomplement of $q^\perp$, proving $\eqref{so5}\Rightarrow\eqref{so3}$ (a similar argument appears in the proof of \cite{MaedaMaeda1970} Theorem (4.18)).
\end{proof}

\begin{thm}\label{anncentre}
For any $B\in\mathscr{P}(A)^\perp$, the following are equivalent in $\mathscr{P}(A)^\perp$.
\begin{enumerate}
\item\label{ac1} $B$ is an ideal.
\item\label{ac2} $B$ is central.
\item\label{ac3} $B$ is $\vee$-distributive.
\item\label{ac4} $B$ is/has a $\wedge$-pseudocomplement.
\item\label{ac5} $B$ is/has a unique complement.
\end{enumerate}
\end{thm}

\begin{proof} \eqref{ac2}$\Rightarrow$\eqref{ac3},\eqref{ac4},\eqref{ac5} is immediate.  The other implications are proved as follows.
\begin{itemize}
\item[\eqref{ac1}$\Rightarrow$\eqref{ac2}] See \cite{Bice2014c} Corollary 5.2.
\item[\eqref{ac2}$\Rightarrow$\eqref{ac1}] If $B\in\mathscr{P}(A)^\perp$ is not an ideal then $B^{\perp\perp}=B\subsetneqq ABA$ and thus we have $a\in A\setminus\{0\}$ with $a^*a\in B$ and $aa^*\in B^\perp$.  Define $u=u_a$ and $v=v_a$ which, as $\mathscr{P}(A)^\perp$ is $\leq$-dense, means we have $C\in\mathscr{P}(A)^\perp\setminus\{0\}$ with $C\subseteq vA_av^*$.  As $C\subseteq A_a$,
\[B^\perp\cap C=B^\perp\cap A_a\cap C=\overline{aAa^*}\cap C\subseteq\overline{aAa^*}\cap vA_av^*.\]  But, as shown in the proof of \autoref{Sasaki}, $vv^*uu^*vv^*=\frac{1}{2}vv^*$ so $||vv^*uu^*||=\frac{1}{\sqrt{2}}<1$ and hence $vA_av^*\cap\overline{aAa^*}=\{0\}$.  Likewise
\[B\cap C=B\cap A_a\cap C=\overline{a^*Aa}\cap C\subseteq\overline{a^*Aa}\cap vA_av^*=\{0\},\] so $C\neq\{0\}=(C\cap B)\vee(C\cap B^\perp)$ and hence $B$ is not central.
\item[\eqref{ac3}$\Rightarrow$\eqref{ac2}] If $B$ is $\vee$-distributive then $B^\perp$ is $\vee$-distributive and hence central, by \autoref{so}, so $B$ is central too.
\item[\eqref{ac4}$\Rightarrow$\eqref{ac1}] Say $B$ is a $\wedge$-pseudocomplement of $C$ in $\mathscr{P}(A)^\perp$.  As $\mathscr{P}(A)^\perp$ is a $\bigvee$-dense $\wedge$-sublattice of $\mathcal{H}(A)$, $B$ must also be a $\wedge$-pseudocomplement of $C$ in $\mathcal{H}(A)$ and hence $B=C^\perp$ is an ideal, by \autoref{pseudo}.  Likewise, if $C$ is a $\wedge$-pseudocomplement of $B$ in $\mathscr{P}(A)^\perp$ then $C=B^\perp$ is an ideal, as is $B=B^{\perp\perp}$.
\item[\eqref{ac5}$\Rightarrow$\eqref{ac1}] If $B=B^{\perp\perp}$ is not an ideal, then neither is $B^\perp$ and hence there exists $a\in A^1_+$ that does not commute with $p_{B^\perp}$.  Then, as in the proof of \autoref{unitunicomp}, we have a unitary $u\in\mathcal{M}(A)$ with $0<||p_{B^\perp}-u^*p_{B^\perp}u||<1$, and hence $C=u^*B^\perp u$ is another complement of $B$ in $\mathscr{P}(A)^\perp$.  So if $B$ has a unique complement, which must be $B^\perp$, then $B$ is an ideal.  Likewise, if $B$ is not an ideal then there exists $a\in A^1_+$ that does not commute with $p_B$, which allows us to find another complement of $B^\perp$ in $\mathscr{P}(A)^\perp$.  So if $B$ is the unique complement of $C$ in $\mathscr{P}(A)^\perp$ then $B=C^\perp$ so $C=B^\perp$ and hence $B$ is an ideal.
\end{itemize}
\end{proof}

A purely order theoretic proof of \eqref{ac5}$\Rightarrow$\eqref{ac2} above would be possible (see \cite{MaedaMaeda1970} Theorem (4.20)) if we could show that $\mathscr{P}(A)^\perp$ is not only SSC, but actually SC.

\begin{dfn}
$\mathbb{P}$ is \emph{section complemented (SC)} if $[0,p]$ is complemented, for $p\in\mathbb{P}$.
\end{dfn}

\begin{qst}
Is $\mathscr{P}(A)^\perp$ section complemented?
\end{qst}

We also have the following analog of \autoref{HAcomm}.  As in \cite{Bice2014c} Definition 5.5, call $B\in\mathscr{P}(A)^\perp$ \emph{$\triangledown$-finite} when $C\subseteq B$ and $C^\triangledown=B^\triangledown$ implies $C=B$, for all $C\in\mathscr{P}(A)^\perp$.  We also define $C^{\perp_B}=C^\perp\cap B$ and $C^{\triangledown_B}=C^\triangledown\cap B$.

\begin{cor}\label{Acomm}
For any $B\in\mathscr{P}(A)^\perp$, the following are equivalent.
\begin{enumerate}
\item\label{Acomm1} $B$ is commutative.
\item\label{Acomm2} $\mathscr{P}(B)^{\perp_B}=\mathscr{P}(B)^{\triangledown_B}$.
\item\label{Acomm3} $B$ is $\triangledown$-finite.
\item\label{Acomm4} $\mathscr{P}(A)^{\perp_B}$ is distributive.
\item\label{Acomm5} $B=C^{\perp\perp}$ for some commutative $C\in\mathcal{H}(A)$.
\end{enumerate}
Moreover, if any/all of these conditions is satisfied then $\mathscr{P}(A)^{\perp_B}=\mathscr{P}(B)^{\perp_B}$.
\end{cor}

\begin{proof} If $B$ is commutative then $\mathscr{P}(A)^{\perp_B}=\mathscr{P}(B)^{\perp_B}$, by \cite{Bice2014c} Theorem 5.4.
\begin{itemize}
\item[\eqref{Acomm1}$\Rightarrow$\eqref{Acomm2}] See \autoref{HAcomm} \eqref{HAcomm1}$\Rightarrow$\eqref{HAcomm2} and \cite{Bice2014c} Theorem 4.6.
\item[\eqref{Acomm2}$\Rightarrow$\eqref{Acomm3}] See \autoref{HAcomm} \eqref{HAcomm2}$\Rightarrow$\eqref{HAcomm3} and \cite{Bice2014c} Theorem 5.3.
\item[\eqref{Acomm1}$\Rightarrow$\eqref{Acomm4}] Use $\mathscr{P}(A)^{\perp_B}=\mathscr{P}(B)^{\perp_B}$, \eqref{Acomm1}$\Rightarrow$\eqref{Acomm2} and \cite{Bice2014c} Corollary 5.2.
\item[\eqref{Acomm4}$\Rightarrow$\eqref{Acomm2}] If $\mathscr{P}(B)^{\perp_B}\neq\mathscr{P}(B)^{\triangledown_B}$ then, as in \autoref{anncentre}, take $C\in\mathscr{P}(B)^{\perp_B}\setminus\mathscr{P}(B)^{\triangledown_B}$ and $D\in\mathscr{P}(B)^{\perp_B}\setminus\{0\}$ such that $C\cap D=\{0\}=C^{\perp_B}\cap D$ and $D$ is in the hereditary C*-subalgebra generated by $C$ and $C^{\perp_B}$, so $D\subseteq C\vee C^{\perp_B}$ (with the supremum taken in $\mathscr{P}(A)^\perp$ \textendash\, as we may have $\mathscr{P}(A)^{\perp_B}\neq\mathscr{P}(B)^{\perp_B}$, we may have $C\vee C^{\perp_B}<B$ so this does not follow automatically from $D\subseteq B$).  Thus $D\cap(C\vee C^{\perp_B})=D\neq\{0\}=(D\cap C)\vee(D\cap C^{\perp_B})$ so $\mathscr{P}(A)^{\perp_B}$ is not distributive.

\item[\eqref{Acomm3}$\Rightarrow$\eqref{Acomm2}] If $\mathscr{P}(B)^{\perp_B}\neq\mathscr{P}(B)^{\triangledown_B}$ then, taking $C\in\mathscr{P}(B)^{\perp_B}\setminus\mathscr{P}(B)^{\triangledown_B}$, we have $C\subsetneqq C^{\triangledown_B\triangledown_B}$ and hence $C\vee C^{\triangledown_B}\subsetneqq B$, as $C^{\triangledown_B\triangledown_B}$ is central in $\mathscr{P}(B)^{\perp_B}$, even though $(C\vee C^{\triangledown_B})^{\triangledown_B\triangledown_B}=(C^{\triangledown_B}\cap C^{\triangledown_B\triangledown_B})^{\triangledown_B}=\{0\}^{\triangledown_B}=B$.  By \cite{Bice2014c} Theorem 5.3, $(C\vee C^{\triangledown_B})^{\triangledown_B\triangledown_B}=(C\vee C^{\triangledown_B})^{\triangledown\triangledown}\cap B$ so $B\subseteq(C\vee C^{\triangledown_B})^{\triangledown\triangledown}$ and hence, as $C\vee C^{\triangledown_B}\subseteq B$, we have $(C\vee C^{\triangledown_B})^{\triangledown\triangledown}=B^{\triangledown\triangledown}$ so $B$ is not $\triangledown$-finite.

\item[\eqref{Acomm2}$\Rightarrow$\eqref{Acomm1}] By \eqref{annproj}, any $b\in B$ can be approximated arbitrarily closely by linear combinations of open projections corresponding to *-annihilators of $B$.  If $\mathscr{P}(B)^{\perp_B}=\mathscr{P}(B)^{\triangledown_B}$ then each one of these projections is in $B'$ so $B\subseteq B'$.

\item[\eqref{Acomm1}$\Rightarrow$\eqref{Acomm5}] Immediate.

\item[\eqref{Acomm5}$\Rightarrow$\eqref{Acomm1}] We first claim $B\subseteq C'$.  If not, we would have $c\in C^1_+$ and $b\in B_+$ such that $bc\neq cb$.  Then, for some $\epsilon>0$, we must have $bc_{[\epsilon,1]}\neq c_{[\epsilon,1]}b$ and hence $c_{[\epsilon,1]}b(1-c_{[\epsilon,1]})=c_{[\epsilon,1]}bc_{[0,\epsilon)}\neq0$.  Thus, for some $\delta<\epsilon$ sufficiently close to $\epsilon$, we must have $c_{[\epsilon,1]}bc_{[0,\delta]}\neq0$ and hence $f(c)bg(c)\neq0$ where
$f$ and $g$ are continuous functions on $[0,1]$, $f[0,(\epsilon+\delta)/2]=\{0\}=g[(\epsilon+\delta)/2,1]$ and $f[\epsilon,1]=\{1\}=g[0,\delta]$.  If we had $g(c)bf(b)^2bg(c)\in C^\perp$ then, as $b\in B$, $f(c)bg(c)b=0$ and hence $f(c)bg(c)=0$, a contradiction.  Thus $f(c)bg(c)a\neq0$ for some $a\in C$.  As $C$ is hereditary and both $f(c)$ and $a$ are in $C$, this means that $d=f(c)bg(c)a\in C$ and, likewise $d^*\in C$.  However, $dd^*\leq\lambda f(c)^2$ for some $\lambda>0$ while $d^*d\leq\lambda'g(c)^2$ (note that $a,c\in C$ so $a$ commutes with $c$ and hence with $g(c)\in C+\mathbb{C}1$) for some $\lambda'>0$.  As $f(c)g(c)=0$, this means that $d$ and $d^*$ do not commute, contradicting the fact $C$ is commutative.

Now the claim is proved, take any $a,b\in B_+$.  Given any $c\in C_+$, note that $c(ab-ba)=c^{1/4}ac^{1/2}bc^{1/4}-c^{1/4}bc^{1/2}ac^{1/4}=0$, as $c^{1/4}ac^{1/4},c^{1/4}bc^{1/4}\in C$.  Thus $ab-ba\in C^\perp\cap B=\{0\}$ and hence, as $a$ and $b$ were arbitrary, $B$ is commutative.

\end{itemize}
\end{proof}

\section{C*-Algebra Type Decompositions}\label{C*TD}

Type classification and decomposition has played a fundamental role in von Neumann algebra theory since its inception almost a century ago.  Somewhat analogous type classifications/decompositions have also been obtained for more general C*-algebras, for example in \cite{Cuntz1977} and \cite{CuntzPedersen1979}.  However, it is only recently that completely consistent extensions of the original von Neumann algebra type decomposition have been obtained by utilizing *-annihilators, either explicitly, as in \cite{Bice2014c}, or implicitly, as in \cite{NgWong2013}.\footnote{A major stumbling block appears to have been the appropriate extension of the type I concept from von Neumann algebras to C*-algebras.  In \cite{CuntzPedersen1979} (and earlier in \cite{Glimm1961}), a C*-algebra is called type I (equivalently, postliminary or GCR) when it has only type I representations.  But even type I von Neumann algebras can have non-type I representations, so this does not encapsulate the original meaning of type I.  It is rather the concept of a \emph{discrete} C*-algebra, introduced in \cite{PeligradZsido2000}, that consistently extends the notion of a type I von Neumann algebra.}  In this section, we outline how to obtain order theoretic type decompositions of $A$ and what algebraic characterizations these types have.

First note that, by definition, central elements $\mathbb{P}^\mathrm{C}$ in $\mathbb{P}$ lead to finite direct product decompositions.  To extend this to infinite products requires separativity.  Indeed, we need separativity to first show that the centre of a complete ortholattice $\mathbb{P}$ is a complete sublattice of $\mathbb{P}$, by \cite{MaedaMaeda1970} Corollary (5.14) (although in the case of $\mathscr{P}(A)^\perp$, we know that the centre is $\mathscr{P}(A)^\triangledown$, by \autoref{anncentre}, which can be easily verified to be a complete sublattice of $\mathscr{P}(A)^\perp$ directly - see \cite{Bice2014c}).  Then we can define the \emph{central cover} $\mathrm{c}(p)$ of $p\in\mathbb{P}$ by
\[\mathrm{c}(p)=\bigwedge[p,1]\cap\mathbb{P}^\mathrm{C}.\]
We now get infinite product decompositions as follows.

\begin{thm}\label{sepprod}
If $(p_\alpha)\subseteq\mathbb{P}$ and $\mathrm{c}(p_\alpha)\wedge\mathrm{c}(p_\beta)=0$, for $\alpha\neq\beta$, $[0,\bigvee p_\alpha]\cong\prod[0,p_\alpha]$.
\end{thm}

\begin{proof}
See \cite{MaedaMaeda1970} Lemma 5.8 and Corollary 5.14.
\end{proof}

Now say we have some class $\mathbf{L}$ of lattices which is closed under infinite direct products, factors and isomorphisms (which is called a \emph{type class} in \cite{FoulisPulmannova2010}, in the slightly different context of effect algebras).  Then \autoref{sepprod} means that, \[\{p\in\mathbb{P}:[0,p]\in\mathbf{L}\}\] is $\mathbb{P}^\mathrm{C}$-complete, according to \cite{Bice2014b} Definition 2.2.  We then get the following type decomposition from \cite{Bice2014b} Theorem 2.6 (see also \cite{Bice2014b} Theorem 2.4).

\begin{thm}\label{typedecomp}
There exists a unique $p\in\mathbb{P}^\mathrm{C}$ such that $p=\mathrm{c}(q)$, where $[0,q]\in\mathbf{L}$, and $[0,r]\notin\mathbf{L}$, for all $r\in(0,p^\perp]$.
\end{thm}

In particular, we can take $\mathbb{P}=\mathscr{P}(A)^\perp$ and, to get decompositions like in the original von Neumann algebra type decompositions (see \cite{MurrayvonNemann1936}), we can take $\mathbf{L}$ to be a class exhibiting some degree of distributivity.

To start with, let $\mathbf{L}$ be the class of distributive lattices.  By \autoref{Acomm}, this decomposition agrees with that given in \cite{Bice2014c} Theorem 5.7.  Moreover, the $B\in\mathscr{P}(A)^\perp$ corresponding to the $p$ in \autoref{typedecomp} is \emph{discrete}, according to \cite{PeligradZsido2000} Definition 2.1, while $B^\perp$ is \emph{antiliminary}, according to \cite{Pedersen1979} 6.1.1.  Thus, this decomposition also agrees with the $A_\mathrm{d}$ vs $A_\mathrm{II}+A_\mathrm{III}$ part of the decomposition in \cite{NgWong2013} Theorem 5.2.  Rephrasing \autoref{typedecomp} in this case, we have the following.

\begin{thm}
There exists unique $B,C\in\mathscr{P}(A)^\perp$ with $B$ discrete, $C$ antiliminary and $A=B\vee C$.
\end{thm}

When $A$ is a von Neumann algebra, the $B$ above is the type I part of $A$, while $C$ is the type II/III part, so this really is completely consistent with the original von Neumann algebra type I vs II/III decomposition.  The only key difference between these kinds of decompositions in the von Neumann vs general C*-algebra case is that the supremum $\vee$ here may not correspond to an algebraic direct sum $\oplus$ in general, i.e. we may have $A\neq B\oplus C$, although we do necessarily have $A=(B\oplus C)^{\perp\perp}$, i.e. $B\oplus C$ will be an essential ideal in $A$.

We can also consider \autoref{typedecomp} when $\mathbb{P}=\mathscr{P}(A)^\perp$ and $\mathbf{L}$ is the larger class of \emph{modular} lattices, i.e. satisfying
\[p\leq r\quad\Rightarrow\quad p\vee(q\wedge r)=(p\vee q)\wedge r.\]
Then the $B\in\mathscr{P}(A)^\perp$ corresponding to the $p$ in \autoref{typedecomp} is, when $A$ is a von Neumann algebra, precisely the type I/II part of $A$, while $B^\perp$ is the type III part of $A$, by the theorem at the start of \cite{Kaplansky1955}.  In this case it also coincides with the decomposition obtained in \cite{Bice2014c} Theorem 6.10 using the relation $\sim$ on *-annihilators (which coincides with Murray-von Neumann equivalence of projections in the von Neumann algebra case), and with the $A_\mathrm{d}+A_\mathrm{II}$ vs $A_\mathrm{III}$ part of the decomposition obtained in \cite{NgWong2013} Theorem 5.2 using the Cuntz-Pedersen equivalence relation on $A_+$.  It would seem plausible that the decomposition based on modularity agrees with that based on the $\sim$ relation on *-annihilators even in more general C*-algebras and, indeed, this would follow if a converse to \cite{Bice2014c} Theorem 6.14 could be proved.  However, UHF algebras are purely infinite with respect to the $\sim$ relation on *-annihilators (see \cite{Bice2013} Proposition 3.89), but finite with respect to the Cuntz-Pedersen equivalence relation on $A_+$ (as UHF algebras have a faithful trace), so these decompositions do not agree in this case.

We can also consider a slight variant of \autoref{typedecomp} when $\mathbb{P}=\mathscr{P}(A)^\perp$ and $\mathbf{L}$ is the class of \emph{orthomodular} lattices, i.e. those ortholattices satisfying
\[p\leq q\quad\Rightarrow\quad p\vee(p^\perp\wedge q)=q.\]
In this case the lattice $[0,q]$ in \autoref{typedecomp} must be replaced with the ortholattice $[0,q]^{\perp_q}=\{r^\perp\wedge q:r\leq q\}$ and likewise for $[0,r]$ (and this ortholattice variant of \autoref{typedecomp} must be obtained from a similar ortholattice variant of \autoref{sepprod}).  In the von Neumann algebra case this is not very interesting, as *-annihilators correspond to projections and projections are always orthomodular, i.e. $p=1$ in this case.  But *-annihilators in an arbitrary C*-algebra may not be orthomodular, as \autoref{C01K} shows, so $p^\perp$ may be non-zero in this case and one might naturally call this $p^\perp$ the `type IV' part of $A$.  But even though *-annihilators in \autoref{C01K} are not orthomodular, the C*-algebra in \autoref{C01K} still contains a full orthomodular (even commutative) *-annihilator, so it is not type IV (in fact, it is even type I in the restrictive C*-algebra sense of having only type I representations).  Thus this example does not answer the following question.

\begin{qst}
Do there exist any (non-zero) type IV C*-algebras?  I.e. do there exist C*-algebras $A$ for which $\mathscr{P}(B)^\perp$ is not orthomodular for any $B\in\mathscr{P}(A)^\perp$?
\end{qst}

\bibliography{maths}{}

\begin{thebibliography}{BRVdB89}

\bibitem[AE02]{AkemannEilers2002}
Charles~A. Akemann and S{\o}ren Eilers.
\newblock Regularity of projections revisited.
\newblock {\em J. Operator Theory}, 48(3, suppl.):515--534, 2002.

\bibitem[Ake69]{Akemann1969}
Charles~A. Akemann.
\newblock The general {S}tone-{W}eierstrass problem.
\newblock {\em J. Functional Analysis}, 4:277--294, 1969.

\bibitem[Ake70]{Akemann1970}
Charles~A. Akemann.
\newblock Left ideal structure of {C}*-algebras.
\newblock {\em J. Functional Analysis}, 6:305--317, 1970.
\newblock \href {http://dx.doi.org/10.1016/0022-1236(70)90063-7}
  {\path{doi:10.1016/0022-1236(70)90063-7}}.

\bibitem[Ake71]{Akemann1971}
Charles~A. Akemann.
\newblock A {G}elfand representation theory for {$C^{\ast} $}-algebras.
\newblock {\em Pacific J. Math.}, 39:1--11, 1971.

\bibitem[Ber85]{Beran1985}
Ladislav Beran.
\newblock {\em Orthomodular lattices}.
\newblock Mathematics and its Applications (East European Series). D. Reidel
  Publishing Co., Dordrecht, 1985.
\newblock Algebraic approach.

\bibitem[Bic12]{Bice2009}
Tristan Bice.
\newblock The order on projections in {C}*-algebras of real rank zero.
\newblock {\em Bull. Polish Acad. Sci. Math.}, 60(1):37--58, 2012.
\newblock \href {http://dx.doi.org/10.4064/ba60-1-4}
  {\path{doi:10.4064/ba60-1-4}}.

\bibitem[Bic13a]{Bice2013}
Tristan Bice.
\newblock Annihilators and type decomposition in {C}*-algebras.
\newblock 2013.
\newblock \href {http://arxiv.org/abs/1310.4639} {\path{arXiv:1310.4639}}.

\bibitem[Bic13b]{Bice2011}
Tristan Bice.
\newblock Filters in {C}*-algebras.
\newblock {\em Canad. J. Math.}, (3):485--509, 2013.
\newblock \href {http://dx.doi.org/10.4153/CJM-2011-095-4}
  {\path{doi:10.4153/CJM-2011-095-4}}.

\bibitem[Bic13c]{Bice2012}
Tristan Bice.
\newblock The projection calculus.
\newblock {\em M\"unster J. Math.}, 6:557--581, 2013.
\newblock URL: \url{http://wwwmath.uni-muenster.de/mjm/vol_6/mjm_vol_6_17.pdf}.

\bibitem[Bic14a]{Bice2014c}
Tristan Bice.
\newblock *-{A}nnihilators in proper *-semigroups.
\newblock {\em Semigroup Forum}, pages 1--24, 2014.
\newblock \href {http://dx.doi.org/10.1007/s00233-014-9611-2}
  {\path{doi:10.1007/s00233-014-9611-2}}.

\bibitem[Bic14b]{Bice2014b}
Tristan Bice.
\newblock Type decomposition in posets.
\newblock 2014.
\newblock (to appear in Math. Slovaca).
\newblock \href {http://arxiv.org/abs/1403.4172} {\path{arXiv:1403.4172}}.

\bibitem[Bir67]{Birkhoff1967}
Garrett Birkhoff.
\newblock {\em Lattice theory}.
\newblock Third edition. American Mathematical Society Colloquium Publications,
  Vol. XXV. American Mathematical Society, Providence, R.I., 1967.

\bibitem[Bla13]{Blackadar2013}
B.~Blackadar.
\newblock {\em Operator algebras: Theory of {C}*-algebras and von Neumann
  algebras}.
\newblock 2013.
\newblock URL: \url{http://wolfweb.unr.edu/homepage/bruceb/Cycr.pdf}.

\bibitem[Bly05]{Blyth2005}
T.~S. Blyth.
\newblock {\em Lattices and ordered algebraic structures}.
\newblock Universitext. Springer-Verlag London, Ltd., London, 2005.

\bibitem[BRVdB89]{BorceuxRosickyBossche1989}
F.~Borceux, J.~Rosick{\'y}, and G.~Van~den Bossche.
\newblock Quantales and {$C^*$}-algebras.
\newblock {\em J. London Math. Soc. (2)}, 40(3):398--404, 1989.
\newblock \href {http://dx.doi.org/10.1112/jlms/s2-40.3.398}
  {\path{doi:10.1112/jlms/s2-40.3.398}}.

\bibitem[BVdB86]{BorceuxBossche1986}
Francis Borceux and Gilberte Van~den Bossche.
\newblock Quantales and their sheaves.
\newblock {\em Order}, 3(1):61--87, 1986.
\newblock \href {http://dx.doi.org/10.1007/BF00403411}
  {\path{doi:10.1007/BF00403411}}.

\bibitem[CP79]{CuntzPedersen1979}
Joachim Cuntz and Gert Pedersen.
\newblock Equivalence and traces on {C}*-algebras.
\newblock {\em J. Funct. Anal.}, 33(2):135--164, 1979.
\newblock \href {http://dx.doi.org/10.1016/0022-1236(79)90108-3}
  {\path{doi:10.1016/0022-1236(79)90108-3}}.

\bibitem[Cun77]{Cuntz1977}
Joachim Cuntz.
\newblock The structure of multiplication and addition in simple {C}*-algebras.
\newblock {\em Math. Scand.}, 40:215--233, 1977.
\newblock URL: \url{http://www.mscand.dk/article.php?id=2360}.

\bibitem[Eff63]{Effros1963}
Edward~G. Effros.
\newblock Order ideals in a {C}*-algebra and its dual.
\newblock {\em Duke Math. J.}, 30:391--411, 1963.

\bibitem[FP10]{FoulisPulmannova2010}
David~J. Foulis and Sylvia Pulmannov{\'a}.
\newblock Type-decomposition of an effect algebra.
\newblock {\em Found. Phys.}, 40(9-10):1543--1565, 2010.
\newblock \href {http://dx.doi.org/10.1007/s10701-009-9344-3}
  {\path{doi:10.1007/s10701-009-9344-3}}.

\bibitem[Gli61]{Glimm1961}
James Glimm.
\newblock Type {I} {$C^{\ast} $}-algebras.
\newblock {\em Ann. of Math. (2)}, 73:572--612, 1961.
\newblock \href {http://dx.doi.org/10.2307/1970319}
  {\path{doi:10.2307/1970319}}.

\bibitem[HO88]{HladnikOmladic1988}
Milan Hladnik and Matja{\v{z}} Omladi{\v{c}}.
\newblock Spectrum of the product of operators.
\newblock {\em Proc. Amer. Math. Soc.}, 102(2):300--302, 1988.
\newblock \href {http://dx.doi.org/10.2307/2045879}
  {\path{doi:10.2307/2045879}}.

\bibitem[Jak73]{Jakubik1973}
J{\'a}n Jakub{\'{\i}}k.
\newblock Center of a complete lattice.
\newblock {\em Czechoslovak Math. J.}, 23(98):125--138, 1973.

\bibitem[Jan78]{Janowitz1978}
M.~F. Janowitz.
\newblock Note on the center of a lattice.
\newblock {\em Math. Slovaca}, 28(3):235--242, 1978.

\bibitem[Kal83]{Kalmbach1983}
Gudrun Kalmbach.
\newblock {\em Orthomodular lattices}, volume~18 of {\em London Mathematical
  Society Monographs}.
\newblock Academic Press Inc. [Harcourt Brace Jovanovich Publishers], London,
  1983.

\bibitem[Kap55]{Kaplansky1955}
Irving Kaplansky.
\newblock Any orthocomplemented complete modular lattice is a continuous
  geometry.
\newblock {\em Ann. of Math. (2)}, 61:524--541, 1955.
\newblock \href {http://dx.doi.org/10.2307/1969811}
  {\path{doi:10.2307/1969811}}.

\bibitem[KD08]{DudekKondo2008}
Michiro Kondo and Wieslaw~A. Dudek.
\newblock On bounded lattices satisfying {E}lkan's law.
\newblock {\em Soft Computing}, 12(11):1035--1037, 2008.
\newblock \href {http://dx.doi.org/10.1007/s00500-007-0270-z}
  {\path{doi:10.1007/s00500-007-0270-z}}.

\bibitem[Kun80]{Kunen1980}
Kenneth Kunen.
\newblock {\em Set theory: An introduction to independence proofs}, volume 102
  of {\em Studies in Logic and the Foundations of Mathematics}.
\newblock North-Holland Publishing Co., Amsterdam, 1980.

\bibitem[Mac64]{MacLaren1964}
M.~Donald MacLaren.
\newblock Atomic orthocomplemented lattices.
\newblock {\em Pacific J. Math.}, 14:597--612, 1964.
\newblock URL: \url{http://projecteuclid.org/euclid.pjm/1103034188}.

\bibitem[MM70]{MaedaMaeda1970}
F.~Maeda and S.~Maeda.
\newblock {\em Theory of symmetric lattices}.
\newblock Die Grundlehren der mathematischen Wissenschaften, Band 173.
  Springer-Verlag, New York, 1970.

\bibitem[Mul86]{Mulvey1986}
Christopher~J. Mulvey.
\newblock \&.
\newblock {\em Rend. Circ. Mat. Palermo (2) Suppl.}, (12):99--104, 1986.
\newblock Second topology conference (Taormina, 1984).

\bibitem[MvN36]{MurrayvonNemann1936}
F.~J. Murray and J.~v.~Neumann.
\newblock On rings of operators.
\newblock {\em Ann. of Math. (2)}, 37(1):116--229, 1936.
\newblock \href {http://dx.doi.org/10.2307/1968693}
  {\path{doi:10.2307/1968693}}.

\bibitem[NW13]{NgWong2013}
Chi-Keung Ng and Ngai-Ching Wong.
\newblock On the decomposition into discrete, type {II} and type {III}
  {C}*-algebras.
\newblock 2013.
\newblock \href {http://arxiv.org/abs/1310.5464} {\path{arXiv:1310.5464}}.

\bibitem[Oza14]{Ozawa2014}
Narutaka Ozawa.
\newblock Can non-central projections still commute with all other projections?
\newblock MathOverflow, 2014.
\newblock URL: \url{http://mathoverflow.net/q/155811}.

\bibitem[Ped79]{Pedersen1979}
Gert~K. Pedersen.
\newblock {\em {$C^{\ast} $}-algebras and their automorphism groups}, volume~14
  of {\em London Mathematical Society Monographs}.
\newblock Academic Press Inc. [Harcourt Brace Jovanovich Publishers], London,
  1979.

\bibitem[Phi04]{Phillips2004}
N.~Christopher Phillips.
\newblock A simple separable {$C^*$}-algebra not isomorphic to its opposite
  algebra.
\newblock {\em Proc. Amer. Math. Soc.}, 132(10):2997--3005 (electronic), 2004.
\newblock \href {http://dx.doi.org/10.1090/S0002-9939-04-07330-7}
  {\path{doi:10.1090/S0002-9939-04-07330-7}}.

\bibitem[PP12]{PicadoPultr2012}
Jorge Picado and Ale{\v{s}} Pultr.
\newblock {\em Frames and locales}.
\newblock Frontiers in Mathematics. Birkh\"auser/Springer Basel AG, Basel,
  2012.
\newblock Topology without points.
\newblock \href {http://dx.doi.org/10.1007/978-3-0348-0154-6}
  {\path{doi:10.1007/978-3-0348-0154-6}}.

\bibitem[PZ00]{PeligradZsido2000}
Costel Peligrad and L\'{a}szl\'{o} Zsid\'{o}.
\newblock Open projections of {C}*-algebras: comparison and regularity.
\newblock In {\em Operator theoretical methods ({T}imi\c soara, 1998)}, pages
  285--300. Theta Found., Bucharest, 2000.

\bibitem[Tom60]{Tomita1959}
Minoru Tomita.
\newblock Spectral theory of operator algebras. {I}.
\newblock {\em Math. J. Okayama Univ.}, 9:63--98, 1959/1960.
\newblock URL:
  \url{http://www.math.okayama-u.ac.jp/mjou/mjou1-46/mjou_pdf/mjou_09/mjou_09_063.pdf}.

\bibitem[Wea03]{Weaver2003}
Nik Weaver.
\newblock A prime {$C^*$}-algebra that is not primitive.
\newblock {\em J. Funct. Anal.}, 203(2):356--361, 2003.
\newblock \href {http://dx.doi.org/10.1016/S0022-1236(03)00196-4}
  {\path{doi:10.1016/S0022-1236(03)00196-4}}.

\end{thebibliography}
\bibliographystyle{alphaurl}

\end{document}